\documentclass[english]{scrartcl}

\usepackage[T1]{fontenc}
\usepackage[utf8]{inputenc}
\usepackage[english]{babel}
\usepackage{lmodern}
\usepackage{exscale}
\usepackage[babel]{microtype}
\usepackage{amsmath, amssymb, mathtools, mathrsfs}
\usepackage{xparse}
\usepackage[inline]{enumitem}

\usepackage{tikz}
\usetikzlibrary{cd,graphs}
\tikzset{mytree/.style={grow'=up, level distance=6mm}, sibling distance=4mm}

\usepackage{csquotes}
\usepackage[style=numeric, sorting=nyt, maxnames=10, maxalphanames=4]{biblatex}

\DeclareNolabel{\nolabel{\regexp{[\p{Z}\p{P}\p{S}\p{C}]+}}}
\usepackage[numbered]{bookmark}
\usepackage{hyperref}
\usepackage{amsthm, thmtools}
\usepackage[xspace]{ellipsis}
\usepackage[all]{hypcap}

\hypersetup{
  colorlinks,
  urlcolor = {blue!60!black},
  linkcolor = {red!60!black},
  citecolor = {green!40!black}
}

\numberwithin{figure}{section}
\numberwithin{equation}{section}
\numberwithin{table}{section}

\numberwithin{mycounter}{section}
\NewDocumentCommand{\declthm}{m m O{plain}}{
  \declaretheorem[
  name = #2,
  sibling = mycounter,
  style = #3
  ]
  {#1}
}
\declthm{theorem}{Theorem}
\declthm{proposition}{Proposition}
\declthm{corollary}{Corollary}
\declthm{lemma}{Lemma}
\declthm{definition}{Definition}[definition]
\declthm{example}{Example}[remark]
\declthm{remark}{Remark}[remark]
\declthm{conjecture}{Conjecture}

\newcommand{\K}{\Bbbk}
\newcommand{\Q}{\mathbb{Q}}
\newcommand{\R}{\mathbb{R}}
\newcommand{\Z}{\mathbb{Z}}

\newcommand{\PP}{\mathsf{P}}
\newcommand{\QQ}{\mathsf{Q}}

\newcommand{\CC}{\mathsf{C}}

\newcommand{\II}{\mathsf{I}}
\newcommand{\Com}{\mathsf{Com}}
\newcommand{\Lie}{\mathsf{Lie}}
\newcommand{\Ass}{\mathsf{Ass}}
\newcommand{\Pois}{\mathsf{Pois}}
\newcommand{\Disk}{\mathsf{Disk}}
\newcommand{\LL}{\mathsf{L}}
\newcommand{\U}{\mathcal{U}}

\newcommand{\GG}{\mathsf{G}}
\newcommand{\ashk}{\textnormal{¡}}
\newcommand{\uni}{{\makebox[1ex]{\tikz[scale=0.25]{ \draw (0,0) -- (0,1); \fill (0,1) circle (0.2); }}}}

\newcommand{\st}[1]{\mathop{\star_{#1}}}
\newcommand{\hst}[1]{\mathop{\hat{\star}_{#1}}}

\DeclareMathOperator{\id}{id}
\DeclareMathOperator{\Tw}{Tw}
\DeclareMathOperator{\Free}{Free}
\DeclareMathOperator{\Hom}{Hom}
\DeclareMathOperator{\im}{im}
\newcommand{\fr}{\mathrm{fr}}

\author{Najib Idrissi\thanks{Université Paris Cité and Sorbonne Université, CNRS, IMJ-PRG, F-75013 Paris, France. Email: najib.idrissi-kaitouni@u-paris.fr}}

\title{Curved Koszul duality of algebras over unital versions of binary operads}

\date{August 22, 2022}

\begin{document}

\maketitle

\begin{abstract}
  We develop a curved Koszul duality theory for algebras presented by quadratic-linear-constant relations over unital versions of binary quadratic operads.
  As an application, we study Poisson $n$-algebras given by polynomial functions on a standard shifted symplectic space.
  We compute explicit resolutions of these algebras using curved Koszul duality.
  We use these resolutions to compute derived enveloping algebras and factorization homology on parallelized simply connected closed manifolds with coefficients in these Poisson $n$-algebras.
\end{abstract}

\tableofcontents

\section{Introduction}

Koszul duality was initially developed by Priddy~\cite{Priddy1970} for associative algebras.
Given an augmented associative algebra $A$, there is a Koszul dual (dg-)algebra $A^{!}$, and there is an equivalence (subject to some conditions) between parts of the derived categories of $A$ and $A^{!}$.
The Koszul dual $A^{!}$ is actually the linear dual of a certain coalgebra $A^{\ashk}$ up to suspension.
If the algebra $A$ satisfies what is called the Koszul property, then the cobar construction of that Koszul dual coalgebra $A^{\ashk}$ is a quasi-free resolution of the algebra $A$.
In this sense, Koszul duality is a tool to produce resolutions of algebras.

An operad governs categories of ``algebras'' in a wide sense, for example associative algebras, commutative algebras, or Lie algebras.
After insights of Kontsevich~\cite{Kontsevich1993a}, Koszul duality was generalized with great success to binary quadratic operads by Ginzburg--Kapranov~\cite{GinzburgKapranov1994} (see also Getzler--Jones~\cite{GetzlerJones1994}), and then to quadratic operads by Getzler~\cite{Getzler1995} (see also~\cite{Fresse2004,Markl1996}).
It was for example realized that the operad governing commutative algebras and the operad governing Lie algebras are Koszul dual to each other.
This duality explains the links between the approaches of Sullivan~\cite{Sullivan1977} and Quillen~\cite{Quillen1969} to rational homotopy theory, which rely respectively on differential graded (dg-) commutative algebras and dg-Lie algebras.
Koszul duality of operads works (mutatis mutandis) like it does for associative algebras.
Given an augmented quadratic operad $\PP$, there is a Koszul dual (dg-)cooperad $\PP^{\ashk}$.
If $\PP$ is Koszul, then the operadic cobar construction of $\PP^{\ashk}$ is a quasi-free resolution of the operad $\PP$.

Operadic Koszul duality was then generalized to several different settings (see Section~\ref{sec:kosz-dual-oper}) and two of them will interest us.
The first, due to Hirsh--Millès~\cite{HirshMilles2012}, is curved Koszul duality applied to (pr)operads with quadratic-linear-constant relations, by analogy with curved Koszul duality for associative algebras~\cite{Positselskiui1993,PolishchukPositselski2005}.
The other, due to Millès~\cite{Milles2012}, is Koszul duality for ``monogenic algebras'' over quadratic operads, a generalization of quadratic algebras over binary operads.
(Note that despite the name, the monogenic algebras defined in~\cite{Milles2012} are not generated by a single element but are rather special kinds of presentations, see Section~\ref{sec:kosz-dual-quadr}.)

Our aim will be to combine, in some sense, the approaches of Millès~\cite{Milles2012} and Hirsh--Millès~\cite{HirshMilles2012} to develop a curved Koszul duality theory for algebras with quadratic-linear-constant relations over unital versions of binary quadratic operads.
The general philosophy is that, since operads are monoids in the category of symmetric sequences, the results of Hirsh--Millès~\cite{HirshMilles2012} are results about the unital associative (colored) operad, which is itself an operad with QLC relations.
With this point of view, we reuse the ideas of Millès~\cite{Milles2012} to define curved Koszul duality for algebras over any operad.

Let us now state our theorem.
We will work over unital versions of quadratic operads that will be made precise in Definition~\ref{def:unital-version}, and our algebras will have quadratic-linear-constant presentations as seen in Definition~\ref{def:alg-qlc}.
The Koszul dual $A^{\ashk}$ of such an algebra $A$ will be defined in Section~\ref{sec:kosz-dual-coalg}, and it will use the quadratic reduction $qA$ of Definition~\ref{def:quad-red}.

\begin{theorem}[Theorem~\ref{thm:main}]
  Let $\PP$ be a binary quadratic operad and $u\PP$ be a unital version of $\PP$.
  Let $A$ be a $u\PP$-algebra with a quadratic-linear-constant presentation and $qA$ be its quadratic reduction and $A^{\ashk} = (qA^{\ashk}, d_{A^{\ashk}}, \theta_{A^{\ashk}})$ be the curved $\PP^{\ashk}$-coalgebra given by its Koszul dual.

  If the $\PP$-algebra $qA$ is Koszul in the sense of Millès~\cite{Milles2012}, then the canonical morphism $\Omega_{\kappa}A^{\ashk} \xrightarrow{\sim} A$ is a quasi-isomorphism of $u\PP$-algebras.
\end{theorem}

Our motivation is that if $\PP$ is a Koszul operad, then there is a functorial way of obtaining resolutions of $\PP$-algebras by considering the bar-cobar construction.
However, this resolution is big, and explicit computations are not always easy.
On the other hand, the theory of Millès~\cite{Milles2012} provides resolutions for Koszul monogenic algebras over Koszul quadratic operads which are much smaller when they exist (cf.\ Remark~\ref{rmk:big}), but the construction is unavailable when the operad is not quadratic and/or when the algebra is not monogenic.
Our theorem allows us to construct resolutions of non-monogenic algebras over non-quadratic operads.

Note that by applying Theorem~\ref{thm:main} theory to different kinds of operads, we recover some already existing notions of ``curved algebras'' and ``Koszul duality of curved algebras''.
For example, when applied to associative algebras, we recover the notion of a curved coalgebra of Lyubashenko~\cite{Lyubashenko2017}, and when applied to Lie algebras, we recover (the dual of) curved Lie algebras~\cite{ChuangLazarevMannan2016,Maunder2017}.

Our main application of Theorem~\ref{thm:main} will be the study of unital Poisson $n$-algebra.
For $n \in \Z$ and $D \ge 0$, consider the Poisson $n$-algebra $A_{n;D} = \R[x_{1}, \dots, x_{D}, \xi_{1}, \dots, \xi_{D}]$, where $\deg x_{i} = 0$, $\deg \xi_{j} = 1-n$ and the shifted Lie bracket is given by $\{x_{i}, \xi_{j}\} = \delta_{ij}$.
We may view $A_{n;D}$ as the algebra $\mathscr{O}_{\mathrm{poly}}(T^{*}\R^{D}[1-n])$ of polynomial functions on the shifted cotangent space of $\R^{D}$, with $x_{i}$ being a coordinate function on $\R^{D}$, and $\xi_{j}$ being the vector field $\partial / \partial x_{j}$.
This algebra admits a presentation with quadratic-linear-constant relations over the operad $u\Pois_{n}$ governing unital Poisson $n$-algebra.
We prove that it is Koszul, and thus that the cobar construction on its Koszul dual $\Omega_{\kappa}A^{\ashk}$ provides a cofibrant replacement of $A$ and we moreover explicitly describe that cofibrant replacement.

We then use $\Omega_{\kappa}A^{\ashk}$ to compute the derived enveloping algebra of $A_{n;D}$, which we prove is quasi-isomorphic to the underived enveloping algebra of $A$.
We also compute the factorization homology $\int_{M} A_{n;D}$ of a simply connected parallelized closed manifold $M$ with coefficients in $A_{n;D}$ and we prove that the homology of $\int_{M} A_{n;D}$ is one-dimensional for such manifolds (Proposition~\ref{prop:final-result}).
This fits in with the physical intuition that the expected value of a quantum observable, which should be a single number, lives in $\int_{M} A$, see e.g., \cite{CostelloGwilliam2017} for a broad reference.

Note that a computation for a similar object was performed by Markarian~\cite{Markarian2017}.
Moreover, shortly after the first version of this paper appeared, Döppenschmitt~\cite{Doeppenschmitt2018} released a preprint containing the computation of the factorization homology of a twisted version of $A_{n;D}$, using physical methods.
See Remark~\ref{rem:markarian}.

\paragraph{Outline}

In Section~\ref{sec:backgr-recoll}, which does not contain any original results, we lay out our conventions and notations as well as background for the rest of the paper.
We give a quick tour of Koszul duality (Section~\ref{sec:kosz-dual-oper}), recall the definition of ``unital version'' of a quadratic operad (Section~\ref{sec:unital-versions}), and give some background on factorization homology (Section~\ref{sec:fact-homol}).
In Section~\ref{sec:curved-coalgebras}, we define curved coalgebras and semi-augmented algebras and bar/cobar constructions suitable for our applications and we prove that they are adjoint to each other.
In Section~\ref{sec:kosz-unit-algebra}, we define algebras with QLC relations and the Koszul dual curved coalgebra of such an algebra, and we prove our main theorem, i.e. that if the quadratic reduction of the algebra is Koszul, then the cobar construction on the Koszul dual of the algebra is a cofibrant replacement of the algebra.
In Section~\ref{sec:appl-sympl-n}, we apply the above theory to the symplectic Poisson $n$-algebras.
We explicitly describe the cofibrant replacement obtained by Koszul duality and we use it to compute their derived enveloping algebras, as well as factorization homology.

\paragraph[]{Acknowledgments}

The author is thankful to Julien Ducoulombier, Benoit Fresse, Thomas Willwacher, and Lukas Woike for helpful discussions and comments as well as to the anonymous referee for helpful remarks and suggestions.
The author was supported by the ERC project StG 678156–GRAPHCPX and ANR project ANR-20-CE40-0016 HighAGT, and contributes to the IdEx University of Paris ANR-18-IDEX-0001.

\section{Conventions, background, and recollections}
\label{sec:backgr-recoll}

We work with $\Z$-graded chain complexes over some base field $\Bbbk$ of characteristic zero, which we call ``dg-modules''.
Given a dg-module $V$, its suspension $\Sigma V$ is given by $(\Sigma V)_{n} = V_{n-1}$.

We work extensively with (co)operads and (co)algebras over (co)operads and we refer to e.g.~\cite{LodayVallette2012} or~\cite[Part~I(a)]{Fresse2017} for a detailed treatment.
Briefly, a (symmetric, one-colored) operad $\PP$ is a collection $\{ \PP(n) \}_{n \ge 0}$ of dg-modules, with each $\PP(n)$ equipped with an action of the symmetric group $\Sigma_{n}$, a unit $\eta \in \PP(1)$, and composition maps $\circ_{i} : \PP(k) \otimes \PP(l) \to \PP(k+l-1)$ for $1 \le i \le k$.
These structural maps satisfy the equivariance, unit, and associativity axioms that can be deduced immediately from the properties of multilinear composition and variable permutation in the prototypical operad $\{\Hom(V^{\otimes n}, V)\}$.
For such an operad $\PP$ and a dg-module $V$, we define $\PP(V) = \bigoplus_{n \ge 0} \PP(n) \otimes_{\Sigma_{n}} V^{\otimes n}$.
A $\PP$-algebra is a dg-module $A$ equipped with a structure map $\gamma_{A} : \PP(A) \to A$ satisfying an associativity axiom.
For a dg-module $V$, we denote by $\PP(V)$ the free $\PP$-algebra on $V$.

Cooperads (usually denoted $\CC$) and coalgebras (usually denoted $C$) are defined dually.
We will generally consider conilpotent cooperads, i.e., cooperads for which iterated cocompositions vanish on any element after enough iterations.
This extra condition is required to have a better behaved category, e.g., to have a much easier description of the cofree cooperad, as well as to avoid certain issues with infinite sums.
For non-conilpotent cooperads and coalgebras, more elaborate techniques are needed, such as the ones recently developed by Roca i Lucio~\cite{Lucio2022}.

Given an operad $\PP$, the suspended operad $\mathscr{S}\PP$ (sometimes denoted $\Lambda^{-1}\PP$) is defined such that a $\mathscr{S}\PP$-algebra structure on $\Sigma A$ is the same thing as a $\PP$-algebra structure on $A$.
The suspended cooperad $\mathscr{S}^{c}\CC$ is defined similarly.
On (co)free (co)algebras, we have $\mathscr{S}\PP(\Sigma V) = \Sigma \PP(V)$ and $\mathscr{S}^{c}\CC(\Sigma V) = \Sigma \CC(V)$.

\begin{example}
  Some examples of operads will appear several times:
  \begin{enumerate*}[label=(\roman*)]
    \item the operad $\Ass$ governing associative algebras;
    \item the operad $\Com$, governing commutative algebras;
    \item the operad $\Lie$, governing Lie algebra;
    \item the operad $\Pois_{n}$ governing Poisson $n$-algebras, i.e., algebras with a commutative product and a Lie bracket of degree $n-1$ which is a biderivation.
    As a symmetric sequence, $\Pois_{n}$ is isomorphic to $\Com \circ \mathscr{S}^{1-n}\Lie$~\cite[Section~13.3.3]{LodayVallette2012}.
    The operad structure is induced by those of $\Com$ and $\Lie$, as well as a distributive law stating that the bracket is a biderivation.
  \end{enumerate*}
\end{example}

If $E = \{ E(n) \}_{n \ge 0}$ is a symmetric sequence, then we write $\Free(E)$ for the free operad generated by $E$.
It can be described in terms of rooted trees with internal vertices decorated by elements of $E$.
Operadic composition is given by grafting of trees.
If $S \subset \PP$ is a subsequence of an operad $\PP$, then we write $\PP / (S)$ for the quotient by the operadic ideal generated by $S$.

\subsection{Koszul duality for{\texorpdfstring{\dots}{...}}}
\label{sec:kosz-dual-oper}

We now briefly recall several incarnations of Koszul duality in order to set up the notations and definitions.
Note that for presentational purposes, we will move on directly from the quadratic case to the quadratic-linear-constant case.
In Priddy's original work~\cite{Priddy1970}, the quadratic-linear case is already covered and applied to Steenrod algebras.

\subsubsection{{\texorpdfstring{\dots}{...}}quadratic associative algebras}
\label{sec:associative-algebras}

Let $A$ be a quadratic associative algebra, i.e., an algebra equipped with a presentation with quadratic relations.
There is an associated Koszul dual coalgebra $A^{\ashk}$ to $A$~\cite{Priddy1970} which can be used to define the Koszul complex $A \otimes_{\kappa} A^{\ashk}$.
By definition, $A$ is Koszul if this Koszul complex is acyclic.
The coalgebra $A^{\ashk}$ is a sub-coalgebra of the bar construction $BA$, and $A$ is Koszul if and only if the inclusion is a quasi-isomorphism.
There is a also canonical morphism from the cobar construction $\Omega A^{\ashk}$ to $A$, and $A$ is Koszul if and only if this canonical morphism is a quasi-isomorphism.
Since the cobar construction $\Omega A^{\ashk}$ is quasi-free as an associative algebra, this allows to produce an small quasi-free resolution of any Koszul algebra.

\begin{example}
  \label{exa:koszul-complex}
  Let $A = \R[x_{1}, \dots, x_{k}]$ be a free commutative algebra on $k$ variables of degree zero.
  The Koszul dual coalgebra $A^{\ashk} = \Lambda^{c}(dx_{1}, \dots, dx_{k})$ is the exterior coalgebra on $k$ variables of degree one.
  The Koszul complex $(A \otimes A^{\ashk}, d_{\kappa})$ has a differential similar to the de Rham differential.
  Since it is acyclic, $A$ is Koszul.
\end{example}

\subsubsection{{\texorpdfstring{\dots}{...}}quadratic operads}
\label{sec:quadratic-operads}

We refer to~\cite{LodayVallette2012} for a detailed treatment.
Let $\PP = \Free(E) / (R)$ be an operad generated $E = \{ E(n) \}_{n \ge 0}$ with relations $R \subset \Free(E)$.
This presentation is quadratic if the relations $R$ is a subsequence of the weight two component $\Free(E)^{(2)}$.
We call $\PP$ quadratic if it admits such a presentation.
One can define the Koszul dual cooperad $\PP^{\ashk}$, the cofree cooperad on cogenerators $\Sigma E$ subject to the corelations $\Sigma^{2} R$ (i.e., $\PP^{\ashk}$ is the smallest sub-cooperad of the cofree cooperad on $\Sigma E$ that contains $\Sigma^2 R$).
When $E$ is finite-dimensional in each arity, one can also define the Koszul dual operad $\PP^{!}$ as the shifted linear dual of $\PP^{\ashk}$.

\begin{example}
  The operad $\Ass = \Ass^!$ is auto-dual.
  The operads $\Com = \Lie^!$ are Koszul dual to each other.
  The operad $\Pois_{n}$ is auto-dual up to suspension, i.e., $\Pois_{n}^{!} = \mathscr{S}^{n-1}\Pois_{n}$.
\end{example}

The cooperad $\PP^{\ashk}$ is a sub-cooperad of the operadic bar construction $B\PP$.
The Rosetta Stone~\cite[Theorem~6.5.7]{LodayVallette2012} implies that morphisms $\PP^\ashk \to B\PP$ are in bijection with twisting morphisms, i.e., equivariant maps $\kappa : \PP^{\ashk} \to \Sigma\PP$ (i.e., of degree $-1$) that satisfy the Maurer--Cartan equation $\partial \kappa + \kappa \star \kappa$.
The operation $\star$ is the preLie convolution product on $\Hom(\PP^{\ashk}, \PP)$, defined on $f,g : \PP^{\ashk} \to \PP$ by:
\begin{equation}\label{eq:star}
  f \star g : \PP^{\ashk} \xrightarrow{\Delta_{(1)}} \PP^{\ashk} \circ_{(1)} \PP^{\ashk} \xrightarrow{f \circ_{(1)} g} \PP \circ_{(1)} \PP \xrightarrow{\gamma_{(1)}} \PP.
\end{equation}
The twisting morphism $\kappa$ induced by $\PP^\ashk \subset B\PP$ is given by:
\begin{equation}
  \label{eq:kappa}
  \kappa : \PP^{\ashk} \twoheadrightarrow \Sigma E \hookrightarrow \Sigma \PP,
\end{equation}
The same Rosetta Stone moreover implies that such twisting morphisms are in bijection with morphisms from the operadic cobar construction $\Omega \PP^{\ashk}$ to $\PP$.
The operad is said to be Koszul if this morphism is a quasi-isomorphism.

The twisting morphism $\kappa : \PP^{\ashk} \to \Sigma \PP$ induces an adjunction $\Omega_{\kappa} \dashv B_{\kappa}$ between the categories of $\PP^{\ashk}$-coalgebras and $\PP$-algebras.
If $\PP$ is Koszul, then $\Omega_{\kappa}B_{\kappa}$ is a functorial cofibrant replacement functor.

\begin{example}
  The operads $\Ass$, $\Com$, $\Lie$, and $\Pois_{n}$ are all Koszul.
  For the operad $\Ass$, $\Omega_\kappa \dashv B_\kappa$ gives the usual bar-cobar resolution.
  For $\Com$, the resolution obtained is (up to degree shifts) the algebra of Chevalley--Eilenberg cochains on the Harrison complex.
\end{example}

\subsubsection{{\texorpdfstring{\dots}{...}}monogenic algebras over operads}
\label{sec:kosz-dual-quadr}

Millès~\cite{Milles2012} extended Koszul duality to monogenic algebras over quadratic operads, a notion which generalizes quadratic associative algebras.
Given a quadratic operad $\PP = \Free(E) / (R)$, a monogenic $\PP$-algebra $A$ is an algebra equipped with a presentation $A = \PP(V) / (S)$, where $S \subset E(V)$ is a set of relations.
To such an algebra, one associates the Koszul dual $\mathscr{S}^c\PP^{\ashk}$-coalgebra $A^{\ashk} = \Sigma \PP^{\ashk}(V, \Sigma S)$, i.e., the suspension of the cofree $\PP^{\ashk}$-coalgebra on $V$ subject to the corelations $\Sigma S$.
There is a canonical algebra-twisting morphism $\varkappa: A^{\ashk} \to \Sigma A$ defined by:
\begin{equation}
  \varkappa : A^{\ashk} \twoheadrightarrow \Sigma V \hookrightarrow \Sigma A,
\end{equation}
Let $\kappa : \PP^{\ashk} \to \Sigma \PP$ be the operad-twisting morphism of Equation~\eqref{eq:kappa}.
The $\kappa$-star product $\star_{\kappa}(\varkappa)$ is given by the composition:
\begin{equation}
  \st{\kappa}(\varkappa) : \Sigma A^{\ashk} \xrightarrow{\Delta_{A^{\ashk}}} \Sigma \PP^{\ashk}(\Sigma^{-1} A^{\ashk}) \xrightarrow{\kappa \circ \varkappa} \Sigma^2 \PP(A) \xrightarrow{\gamma_{A}} \Sigma^2 A.
\end{equation}
The element $\varkappa$ satisfies the Maurer--Cartan equation $\st{\kappa}(\varkappa) = 0$ (the differential vanishes).
It thus defines a morphism $f_{\varkappa} : \Omega_{\kappa}A^{\ashk} \to A$.

The algebra $A$ is said to be Koszul if this morphism $f_{\varkappa}$ is a quasi-isomorphism.
Millès~\cite{Milles2012} proved that this is equivalent to a certain Koszul complex being acyclic, and also equivalent to the adjoint canonical morphism $A^{\ashk} \to BA$ being a quasi-isomorphism.
In this case, the algebra $\Omega_{\kappa}A^{\ashk}$ is an explicit, small resolution of $A$.
For example, if $\PP = \Ass$, this recovers the usual Koszul duality/resolution of associative algebras.

\subsubsection{{\texorpdfstring{\dots}{...}}operads with QLC relations}
\label{sec:curv-kosz-dual}

Curved Koszul duality is a generalization of Koszul duality for \emph{unital} associative algebras~\cite{Positselskiui1993,PolishchukPositselski2005}.
This was generalized by Hirsh--Millès~\cite{HirshMilles2012} for (pr)operads with quadratic-linear-constant (QLC) relations.
See also Priddy~\cite{Priddy1970} for quadratic-linear algebras, Lyubashenko~\cite{Lyubashenko2011} for the unital associative operad, and~\cite{Galvez-CarrilloTonksVallette2012} for operads with quadratic-linear relations.

Let $\II$ be the unit operad, i.e., $\II(1) = \Bbbk$ and $\II(n) = 0$ for $n \neq 1$.
A QLC presentation of an operad is a presentation $\PP = \Free(E) / (R)$, where $E$ is some module of generators and $R \subset \II \oplus E \oplus \Free(E)^{(2)}$ is some module of relations with constant (i.e., multiple of $\id_{\PP}$), linear, and quadratic terms.
The quadratic reduction $q\PP$ is the quadratic operad $\Free(E)/(qR)$, where $qR$ is the projection of $R$ onto $\Free(E)^{(2)}$.
Hirsh--Millès impose some conditions on this presentation: the space of generators is minimal, i.e., $R \cap \langle \II \oplus E \rangle = 0$, and the space of relations is maximal, i.e., $R = (R) \cap \langle \II \oplus E \oplus \Free(E)^{(2)} \rangle$.
Therefore $R$ is the graph of some map $\alpha = (\alpha_{0} + \alpha_{1}) : qR \to \II \oplus E$, i.e., $R = \{ X + \alpha(X) \mid X \in qR \}$.

\begin{example}
  The operad $u\Ass$ governing unital associative algebras has a QLC presentation.
  It is generated by the unit $\uni \in u\Ass(0)$ and the product $\mu \in u\Ass(2)$.
  The relations are $\mu \circ_{1} \mu = \mu \circ_{2} \mu$ (quadratic) and $\mu \circ_{1} \uni = \id = \mu \circ_{2} \uni$ (quadratic-constant).
  Its quadratic reduction $qu\Ass = \Ass \oplus \uni$ is the operad encoding associative algebras $A$ endowed with an element $z_0 \in A$ such that $z_0x = xz_0 = 0$ for all $x \in A$.
\end{example}

From this data, \cite{HirshMilles2012} define the Koszul dual \emph{curved} cooperad $\PP^{\ashk}$, which is a triplet $(q\PP^{\ashk}, d_{\PP^{\ashk}}, \theta_{\PP^{\ashk}})$ where:
\begin{itemize}[nosep]
  \item $q\PP^{\ashk}$ is the Koszul dual cooperad of the quadratic cooperad $q\PP$;
  \item the predifferential $d_{\PP^{\ashk}}$ is the unique degree $-1$ coderivation of $q\PP^{\ashk}$ whose corestriction (composition with the projection) onto $\Sigma E$ is given by $q\PP^{\ashk} \twoheadrightarrow \Sigma^{2} qR \xrightarrow{\alpha_{1}} \Sigma(\Sigma E)$;
  \item the curvature $\theta_{\PP^{\ashk}}$ is the map of degree $-2$ obtained by $q\PP^{\ashk} \twoheadrightarrow \Sigma^{2} qR \xrightarrow{\alpha_{0}} \Sigma^2 \mathsf{I}$.
\end{itemize}

The axioms satisfied by this data imply that the cobar construction $\Omega (q\PP^{\ashk}) = \bigl( \operatorname{Free}(\Sigma^{-1} q\PP^{\ashk}), d_{2} \bigr)$ is equipped with an extra differential $d_{0} + d_{1}$ defined from $\theta_{\PP^{\ashk}}$ and $d_{\PP^{\ashk}}$.
The canonical morphism $\Omega(q\PP^{\ashk}) \to q\PP$ extends to a canonical morphism $\Omega\PP^{\ashk} \coloneqq \bigl( \Omega(q\PP^{\ashk}), d_{0} + d_{1} \bigr) \to \PP$.
If the quadratic operad $q\PP$ is Koszul, then $\Omega\PP^{\ashk} \to \PP$ is a quasi-isomorphism~\cite[Theorem~4.3.1]{HirshMilles2012}.
The operad $\PP$ is therefore said to be Koszul if the quadratic operad $q\PP$ is Koszul in the usual sense.

\begin{remark}
  Le Grignou~\cite{Grignou2017} defined a model category structure on the category of curved cooperads, which is Quillen equivalent to the model category of operads using the bar/cobar adjunction.
\end{remark}

\begin{remark}
  \label{rmk koszul depends presentation}
  It is important to note that, unlike in the quadratic case, the property of being Koszul is not intrinsic to the operad.
  Indeed, every operad $\PP$ admits a Koszul QLC presentation given by $E = \{\mathbf{e}_x \mid x \in \PP\}$ (every element of the operad is a generator) and with quadratic-linear relations given by the ``multiplication table'' $\mathbf{e}_x \circ_i \mathbf{e}_y = \mathbf{e}_{x \circ_i y}$.
  Then $\PP^{\ashk} = B\PP$ is the bar construction of $\PP$, and the cofibrant resolution obtained is the bar-cobar construction.
  However, there exist operads with non-Koszul QLC presentations, e.g., any non-Koszul quadratic operad.
  Being Koszul thus becomes a property of the presentation, rather than of the operad.
\end{remark}

\subsection{Unital versions of operads}
\label{sec:unital-versions}

In what follows, we will only work with algebras over unital versions of binary quadratic operads.
Let $\PP = \operatorname{Free}(E) / (R)$ be a an operad presented by binary generators $E = E(2)$ and quadratic relations $R \subset \bigl( E(2)^{\otimes 2} \bigr)_{\Sigma_{2}}$ and let us assume that the differential of $\PP$ is zero.

\begin{remark}\label{rmk:lem-split}
  While a large part of this paper could be carried out without the assumption that $\PP$ is binary, we need this assumption to be able to prove Lemma~\ref{lem:split} and Proposition~\ref{prop:def-koszul-dual}.
  Proposition~\ref{prop:def-koszul-dual} could be proved for non-binary operads by modifying the weight grading, but not (to the author's knowledge) Lemma~\ref{lem:split}.
\end{remark}

\begin{definition}[{Adapted from \cite[Definition~6.5.4]{HirshMilles2012}}]
  \label{def:unital-version}
  A unital version of $\PP$ is an operad $u\PP$ equipped with a presentation of the form $u\PP = \operatorname{Free}(E \oplus \uni) / (R + R')$, where $\uni$ is a generator of arity zero and degree zero, and such that
  \begin{enumerate}[label=(\roman*)]
    \item the inclusion $E \subset E \oplus \uni$ induces an injective morphism of operads $\PP \to u\PP$;
    \item we have an isomorphism $\uni \oplus \PP \cong qu\PP$ induced by the inclusions $\PP \subset u\PP$ and $\uni \subset u\PP$;
    \item the QLC relations in $R'$ have no linear terms, only quadratic-constant.
  \end{enumerate}
\end{definition}

\begin{example}\label{exa:QLC-operad}
  Examples include:
  \begin{enumerate*}[label=(\roman*)]
    \item $u\Ass$, encoding unital associative algebras;
    \item $u\Com$, encoding unital commutative algebras;
    \item $c\Lie$, encoding Lie algebras equipped with a central element;
    \item $u\Pois_{n}$, encoding Poisson $n$-algebras equipped with an element which is a unit for the product and a central element for the shifted Lie bracket.
  \end{enumerate*}
\end{example}

\subsection{Factorization homology}
\label{sec:fact-homol}

Let us now briefly introduce factorization homology (as in the previous sections, we do not claim any originality).
Factorization homology~\cite{AyalaFrancis2015}, also known as topological chiral homology~\cite{BeilinsonDrinfeld2004}, is an invariant of manifolds with coefficients in a homotopy commutative algebra, just like standard homology is an invariant of topological spaces with coefficients in a commutative ring.
One possible definition of factorization homology is the following~\cite{Francis2013}, which we only give for parallelized manifolds for simplicity.

Fix some integer $n \ge 0$.
Let $\Disk^{\fr}_{n}$ be the endomorphism operad of $\R^{n}$ in the category of parallelized manifolds and embeddings preserving parallelization.
In each arity, we have $\Disk^{\fr}_{n}(k) \coloneqq \operatorname{Emb}^{\fr} \bigl( (\R^{n})^{\sqcup k}, \R^{n} \bigr)$, and operadic composition is given by composition of embeddings.
This operad is weakly equivalent to the usual little $n$-disks operad, i.e., it is an $E_{n}$-operad.
In particular, its homology $H_{*}(\Disk^{\fr}_{n})$ is the unital associative operad $u\Ass$ for $n = 1$, and the unital $n$-Poisson operad $u\Pois_{n}$ for $n \ge 2$ (see Example~\ref{exa:QLC-operad}).

Moreover, given a parallelized $n$-manifold $M$, there is a right $\Disk^{\fr}_{n}$-module given by $\Disk^\fr_{M}(k) \coloneqq \operatorname{Emb}^{\fr} \bigl( (\R^{n})^{\sqcup k}, M \bigr)$.
For a topological $\Disk^{\fr}_{n}$-algebra $A$ (i.e., an $E_{n}$-algebra), the factorization homology of $M$ with coefficients in $A$, denoted by $\int_{M} A$, is given by the derived composition product:
\begin{equation}\label{eq:fact-hom}
  \int_{M} A \coloneqq \Disk^{\fr}_{M} \circ^{\mathbb{L}}_{\Disk^{\fr}_{n}} A = \operatorname{hocoeq} \bigl( \Disk^{\fr}_{M} \circ \Disk^{\fr}_{n} \circ A \rightrightarrows \Disk^{\fr}_{M} \circ A \bigr).
\end{equation}

This definition also makes sense in the category of chain complexes, replacing $\Disk^{\fr}_{n}$ and $\Disk^{\fr}_{M}$ by their respective singular chains chain complexes.
The operad $\Disk_{n}$ is formal over the rationals~\cite{Kontsevich1999,Tamarkin2003,LambrechtsVolic2014,Petersen2014,FresseWillwacher2015}.
Hence, up to homotopy, we may replace $C_{*}(\Disk^{\fr}_{n}; \Q)$ by $H_{*}(\Disk^{\fr}_{n}; \Q) = u\Pois_{n}$ for $n \ge 2$.

In~\cite{Idrissi2016}, given a simply connected closed smooth manifold $M$ with $\dim M \ge 4$, we provided an explicit model for $C_{*}(\Disk^{\fr}_{M}; \R)$.
Our model is a right module over the operad $H_{*}(\Disk^{\fr}_{n}) = u\Pois_{n}$, and the action is compatible with the action of $\Disk^{\fr}_{n}$ on $\Disk^{\fr}_{M}$ in an appropriate sense.
This allows us to compute factorization homology of such manifolds by replacing $C_{*}(\Disk^{\fr}_{M})$ with our model.

This explicit model, denoted $\GG_{P}^{\vee}$, depends on a Poincaré duality model $P$ of $M$, i.e., an (upper-graded) commutative differential graded algebra equipped with a linear form $\varepsilon : P^{n} \to \Q$ satisfying $\varepsilon \circ d = 0$ and inducing a non-degenerate pairing $P^{k} \otimes P^{n-k} \to \Q$, $x \otimes y \mapsto \varepsilon(xy)$ for all $k \in \Z$.
It is moreover a rational model for $M$ in the sense of Sullivan's rational homotopy theory.
These Poincaré duality models exist for all simply connected closed manifolds~\cite{LambrechtsStanley2008}.

We will not give the original definition of $\GG_{P}^\vee$.
Instead, we give the alternative description of~\cite[Section~5]{Idrissi2016}.
Let $\Lie_{n} = \mathscr{S}^{1-n}\Lie$ be the operad governing shifted Lie algebras.
For convenience, let us also define $\LL_{n}(k) \coloneqq \Sigma^{n-1}\Lie_{n}(k)$, which satisfies $\LL_{n}(V) \cong \Lie(\Sigma^{n-1} V)$ for all spaces $V$.
This symmetric sequence is a Lie algebra in the category of right $\Lie_{n}$-modules: the right $\Lie_n$-module of $\Lie_n$ is unaffected by the shift, and the Lie algebra structure is the shift of the canonical left action of $\Lie_n$ on itself.

Given a Lie algebra $\mathfrak{g}$, let its Chevalley--Eilenberg chain complex (with trivial coefficients) be $C_{*}^{CE}(\mathfrak{g}) \coloneqq (\bar{S}^{c}(\Sigma \mathfrak{g}), d_{CE})$, with differential defined by $d_{CE}(x_{1} \dots x_{k}) = \sum_{i < j} \pm x_{1} \dots [x_{i}, x_{j}] \dots \widehat{x_{j}} \dots x_{k}$.
As a right $\Lie_{n}$-module, our model is given by~\cite[Lemma~5.2]{Idrissi2016}:
\begin{equation}
  \GG_{P}^{\vee} \cong_{\Lie_{n}-\mathrm{RMod}} C_{*}^{CE}(P^{-*} \otimes \LL_{n}),
\end{equation}
where $P^{-*}$ is $P$ with grading reversed.
(The right $u\Com$-module structure, which is not explicitly described in~\cite{Idrissi2016}, will be described in Section~\ref{sec:right-ucom-module}.)

Using the theorems of~\cite[Chapter 15]{Fresse2009}, we find that given a $u\Pois_{n}$-algebra $A$, the factorization homology of $M$ with coefficients in $A$ over $\R$ is given, up to quasi-isomorphism and under the hypotheses stated above, by:
\begin{equation}
  \int_M A \simeq \GG_{P}^{\vee} \circ_{u\Pois_{n}}^{\mathbb{L}} A.
\end{equation}

As an example, if $A = S(\Sigma^{1-n} \mathfrak{g})$ is the universal enveloping $n$-algebra of $\mathfrak{g}$, then we recovered in~\cite{Idrissi2016} a theorem of Knudsen~\cite{Knudsen2016} which states that $\int_{M} S(\Sigma^{1-n} \mathfrak{g}) \simeq C_{*}^{CE}(P^{n-*} \otimes \mathfrak{g})$.

\section{Curved bar-cobar adjunction}
\label{sec:curved-coalgebras}

From now on, we fix a binary quadratic operad $\PP = \Free(E) / (R)$ and a unital version $u\PP = \Free(\uni \oplus E) / (R + R')$ as in Section~\ref{sec:unital-versions}, with the same notations as in that section.
Note that $\PP$ is automatically augmented.
We also fix a coaugmented conilpotent binary cooperad $\CC$, with zero differential.
In this section, we define an adjunction between curved $\mathscr{S}^c \CC$-coalgebras and semi-augmented $\PP$-algebras.

Our notion of curved coalgebras depends on a twisting morphism.
We fix throughout the section a twisting morphism $\varphi : \CC \to \Sigma\PP$, i.e., a map of degree $-1$ satisfying the Maurer--Cartan equation $\varphi \star \varphi = 0$, where $\star$ is the convolution product (see Equation~\eqref{eq:star}).
We moreover assume that $\varphi$ vanishes on any element of arity $\geq 3$, and that the underlying map
\begin{equation}
  \varphi : \CC(2) \to \Sigma \PP(2)
\end{equation}
is injective.
This is the case for the Koszul twisting morphism $\kappa : \PP^{\ashk} \to \Sigma \PP$, which is the main example of interest.

\subsection{Curved coalgebras and semi-augmented algebras}
\label{sec:coalg-twist-morph}

Let $C$ be a graded vector space.
For a linear map $\theta : \Sigma^{-1} C \to \Sigma \K \uni = \Sigma u\PP(0)$ (i.e., a linear form of degree $-2$), and for an element $x \in C$, we will denote by $\Theta(x) \in \K$ the scalar such that $\theta(\Sigma^{-1}x) = \Theta(x) \cdot \Sigma\uni$.

\begin{definition}
  \label{def:phi-theta}
  Let $C$ be an $\mathscr{S}^c \CC$-coalgebra (i.e., $\Sigma^{-1} C$ is a $\CC$-coalgebra).
  For any linear map $\theta : \Sigma^{-1} C \to \Sigma \Bbbk \uni$, we define a morphism $\varphi \circ' \theta$ by:
  \begin{multline*}
      \Sigma \CC(2) \otimes_{\Sigma_2} (\Sigma^{-1} C)^{\otimes 2}
       \to \Sigma^2 u\PP(2) \otimes_{\Sigma_2} \bigl( (\Bbbk \uni \otimes C) \oplus (C \otimes \Bbbk \uni) \bigr)
      \\
      \Sigma x(\Sigma^{-1} c_1, \Sigma^{-1} c_2)
       \mapsto
      \Sigma \varphi(x) \bigl( \Theta(c_1) \uni, \, c_2 \bigr) -
      \Sigma \varphi(x) \bigl(c_1, \, \Theta(c_2) \uni \bigr)
      .
  \end{multline*}
  where $\varphi(x) \in \Sigma \PP(2) = \Sigma u\PP(2)$.
\end{definition}

\begin{remark}
  Note that we use the linear isomorphisms $\Sigma \uni \otimes \Sigma^{-1} C \cong \uni \otimes C$ and $\Sigma^{-1} C \otimes \Sigma \uni \cong C \otimes \uni$, which are not $\Sigma_2$-equivariant.
  However, a short computation shows that the previous formula is well-defined on $\Sigma_2$-coinvariants.
\end{remark}

\begin{example}
  Let $\PP = \Ass$, $\CC = \Ass^{\ashk}$, and $\varphi = \kappa : \Ass^{\ashk}(2) \to \Sigma \Ass(2)$ be the Koszul twisting morphism of the associative operad.
  We denote by $\mu \in \Ass(2)$ the generator.
  Let $C$ be a coalgebra over $\Ass^{\vee} = \mathscr{S}^c \Ass^{\ashk}$, i.e., a coassociative coalgebra~\cite[Proposition~9.1.4]{LodayVallette2012}.
  Then we have, for $c_1, c_2 \in C$:
  \begin{equation*}
    \Sigma^2 \mu(\Sigma^{-1} c_1, \Sigma^{-1} c_2) \xmapsto{\varphi \circ' \theta} \Sigma^2\mu(\Theta(c_1) \uni, \, c_2) - \Sigma^2\mu (c_1, \, \Theta(c_2) \uni).
  \end{equation*}
\end{example}

\begin{definition}\label{def:star-phi}
  Let $C$ be an $\mathscr{S}^c\CC$-coalgebra.
  The $\varphi$-star product of a linear map $\theta : \Sigma^{-1} C \to \Sigma \Bbbk \uni$ is the composition:
  \begin{multline}
    \st{\varphi}(\theta) : C \xrightarrow{\Delta_{C}} \Sigma \CC(\Sigma^{-1} C)
    \twoheadrightarrow
    \Sigma \CC(2) \otimes^{\Sigma_2} (\Sigma^{-1} C)^{\otimes 2}
    \xrightarrow{\varphi \circ' \theta}
    \\
    \xrightarrow{\varphi \circ' \theta}
    \Sigma^2 u\PP(2) \otimes_{\Sigma_2} \bigl( (\Bbbk \uni \otimes C) \oplus (C \otimes \Bbbk \uni) \bigr)
    \xrightarrow{\gamma_{u\PP}}
    \Sigma^2 u\PP(C),
  \end{multline}
  where, for the last map, we see $C \cong u\PP(1) \otimes C$ as a subspace of $u\PP(C)$, and we view $\uni$ as an element of $u\PP(0) \subset u\PP(C)$.
\end{definition}

\begin{lemma}\label{lem:split}
  Given a twisting morphism $\varphi : \CC \to \Sigma \PP$, an $\mathscr{S}^c \CC$-coalgebra $C$, and a map $\theta : \Sigma^{-1} C \to \Sigma \Bbbk \uni$, the image of $\st{\varphi}(\theta)$ is included in $\Sigma^2 C \subset \Sigma^2 u\PP(C)$.
\end{lemma}
\begin{proof}
  This follows from conditions (ii) and (iii) of Definition~\ref{def:unital-version}, which imply that the inhomogeneous relations of $u\PP$ have no linear terms.
\end{proof}

\begin{lemma}\label{lem:star-coder}
  Let $C$ be a $\mathscr{S}^c \CC$-coalgebra, $\varphi : \CC \to \Sigma \PP$ be a twisting morphism, and $\theta :\Sigma^{-1} C \to \Sigma \Bbbk \uni$ a linear map.
  The map $\st{\varphi}(\theta) : C \to \Sigma^2 C$ is a coderivation.
\end{lemma}

\begin{proof}
  Let $\Delta^{(r)}_C : C \to \Sigma \CC(r) \otimes_{\Sigma_r} (\Sigma^{-1} C)^{\otimes r}$ be the projection of the structure map of $C$ onto the arity $r$ part.
  Since $\CC$ is binary, we must check that we have:
  \begin{equation}
    \Delta^{(2)}_C \circ (\st{\varphi}(\theta)) = (\st{\varphi}(\theta) \otimes \id_C + \id_C \otimes \st{\varphi}(\theta)) \circ \Delta^{(2)}_C.
  \end{equation}
  Recall that the $\CC$-coalgebra structure on $\Sigma^{-1}C$ is coassociative, therefore we have commutative squares, for $i \in \{1,2\}$ (where suspensions are implicit):
  \begin{equation}
    \begin{tikzcd}[column sep = huge]
      C \ar[r, "\Delta_C^{(2)}"] \ar[d, "\Delta_C^{(3)}"] & \CC(2) \otimes C^{\otimes 2} \ar[d, "\id_{\CC(2)} \otimes \id_C^{\otimes i-1} \otimes \Delta_C^{(2)} \otimes \id_C^{\otimes 2-i}"] \\
      \CC(3) \otimes C^{\otimes 3} \ar[r, "\circ_i^* \otimes \id_C^{\otimes 3}"] & \CC(2) \otimes \CC(2) \otimes C^{\otimes 3}
    \end{tikzcd}
    .
  \end{equation}

  Due to Definition~\eqref{def:phi-theta}, the definition of $\st{\varphi}(\theta)$, and the previous commutative square, in order to prove that $\st{\varphi}(\theta)$ is a coderivation, we need to compare the operations represented graphically by the following trees:
  \begin{equation}
    \Delta^{(2)}_C \circ (\st{\varphi}(\theta)) =
    \begin{tikzpicture}[baseline=4mm]
      \node {$\varphi$} [mytree]
      child { child child }
      child {node{$\theta$}};
    \end{tikzpicture}
    -
    \begin{tikzpicture}[baseline=4mm]
      \node {$\varphi$} [mytree]
      child {node{$\theta$}}
      child { child child }
      ;
    \end{tikzpicture},
  \end{equation}
  \begin{equation}
    (\id_{\CC(2)}(\st{\varphi}(\theta), {\id_C}) + \id_{\CC(2)}({\id_C}, \st{\varphi}(\theta))) \circ \Delta^{(2)}_C
    =
    \begin{tikzpicture}[baseline=4mm]
      \node[inner sep=0pt] {} [mytree]
      child {
          node {$\varphi$}
          child {node{$\theta$}}
          child
        }
      child;
    \end{tikzpicture}
    -
    \begin{tikzpicture}[baseline=4mm]
      \node[inner sep=0pt] {} [mytree]
      child {
          node {$\varphi$}
          child
          child {node{$\theta$}}
        }
      child;
    \end{tikzpicture}
    +
    \begin{tikzpicture}[baseline=4mm]
      \node[inner sep=0pt] {} [mytree]
      child
      child {
          node {$\varphi$}
          child {node{$\theta$}}
          child
        };
    \end{tikzpicture}
    -
    \begin{tikzpicture}[baseline=4mm]
      \node[inner sep=0pt] {} [mytree]
      child
      child {
          node {$\varphi$}
          child
          child {node{$\theta$}}
        };
    \end{tikzpicture}
    ,
  \end{equation}
  where e.g., the first tree represents the map (up to suspensions):
  \begin{multline}
    C \xrightarrow{\Delta^{(3)}_C} \CC(3) \otimes C^{\otimes 3} \xrightarrow{\circ_1^* \otimes \id_C^{\otimes 3}} \CC(2) \otimes \CC(2) \otimes C^{\otimes 3} \\
    \xrightarrow{\varphi \otimes \id_{\CC(2)} \otimes \id_C^{\otimes 2} \otimes \theta} \PP(2) \otimes \CC(2) \otimes C^{\otimes 2} \otimes \K \uni \xrightarrow{\gamma_{u\PP}} u\PP(1) \otimes \CC(2) \otimes C^{\otimes 2} \cong \CC(2) \otimes C^{\otimes 2}.
  \end{multline}

  A similar graphical interpretation of the Maurer--Cartan equation $\varphi \star \varphi = 0$ leads to the vanishing of:
  \begin{equation}
    \begin{tikzpicture}[baseline=4mm]
      \node {$\varphi$} [mytree]
      child {
          node {$\varphi$}
          child child
        }
      child
      ;
    \end{tikzpicture}
    +
    \begin{tikzpicture}[baseline=4mm]
      \node {$\varphi$} [mytree]
      child
      child {
          node {$\varphi$}
          child child
        }
      ;
    \end{tikzpicture}
    :\CC(3) \to (\CC(2) \otimes \CC(2))^{\oplus 2} \to (\PP(2) \otimes \PP(2))^{\oplus 2} \to \PP(3).
  \end{equation}
  Tensoring this relation with $C^{\otimes 3}$ and plugging $\theta$ in the three possible places leads to the coderivation relation using the injectivity of $\varphi : \CC(2) \to \Sigma \PP(2)$ and the associativity and unitality of the structure map of $u\PP$.
\end{proof}

\begin{example}
  Let $u\PP = u\Ass$ be the operad governing unital associative operads and $\kappa : \Ass^{\ashk} \to \Sigma \Ass$ the Koszul twisting morphism.
  Let $C$ be a coassociative coalgebra and let $\theta : \Sigma^{-1} C \to \Sigma \Bbbk$ be a linear map.
  The $\kappa$-star product of $\theta$ is given by:
  \begin{equation}
    \st{\kappa}(\theta) = (\theta \otimes {\id} - {\id} \otimes \theta) \Delta_{C} : C \to C.
  \end{equation}
\end{example}

We introduce the following definition inspired by the definition of a curved coproperad in~\cite[Section~3.2.1]{HirshMilles2012}.
(Informally, we can think of the definition of~\cite{HirshMilles2012} as the case where the twisting morphism $\varphi$ is the Koszul morphism from the colored operad of operads to its Koszul dual.)
\begin{definition}
  \label{def:curved-coalg}
  A $\varphi$-curved $\mathscr{S}^c \CC$-coalgebra is a triple $(C, d_{C}, \theta_{C})$ where:
  \begin{itemize}
    \item $C$ is a $\mathscr{S}^c \CC$-coalgebra (with no differential);
    \item $d_{C} : C \to \Sigma C$ is a coderivation (the ``predifferential'');
    \item $\theta_{C} : \Sigma^{-1} C \to \Sigma \Bbbk \uni$ is a linear map (the ``curvature'');
  \end{itemize}
  satisfying:
  \begin{align}
    d_{C}^{2} & = \st{\varphi}(\theta_{C}), & \theta_{C} d_{C} & = 0.
  \end{align}
\end{definition}

\begin{remark}
  This notion is a generalization of the notion of coalgebra over a curved cooperad from~\cite[Section~5.2.1]{HirshMilles2012}.
  A coalgebra over a curved cooperad $(\CC, d_{\CC}, \theta_{\CC})$ is a pair $(C, d_{C})$ where $C$ is a $\CC$-coalgebra, $d_{C}$ is a coderivation of $C$, and $d_{C}^{2} = (\theta_{\CC} \circ \id_{C})\Delta_{C}$.
  In our setting, the curvature is part of the data of the coalgebra, rather than the cooperad itself, and we have an extra condition $\theta_{C} d_{C} = 0$.
  Moreover, our notion of curved coalgebra depends on the data of a twisting morphism $\varphi : \CC \to \PP$, whereas in~\cite{HirshMilles2012} this is extra data required to define a bar/cobar adjunction.
  Le Grignou~\cite{LeGrignou2016} endowed the category of coalgebras over a curved cooperad with a model category structure, such that the bar/cobar adjunction defines a Quillen equivalence with the category of algebras over non-augmented operads.
\end{remark}

\begin{example}\label{exa:curved-coass}
  Consider $u\PP = u\Ass$, the operad encoding unital associative algebras.
  Then its Koszul dual $\CC = \Ass^{\ashk} = (\mathscr{S}^{c})^{-1}\Ass^{\vee}$ encodes shifted coassociative coalgebras.
  Let $\varphi = \kappa$ be the twisting morphism of Koszul duality.
  A $\kappa$-curved $\mathscr{S}^c \Ass^{\ashk}$-coalgebra is a coassociative coalgebra $C$, equipped with a predifferential $d_{C}$ and a curvature $\theta_{C} : \Sigma^{-1} C \to \Sigma \Bbbk$, satisfying $\theta_{C}d_{C} = 0$ and (compare with Lyubashenko~\cite{Lyubashenko2017}):
  \begin{equation}
    d_{C}^{2} = \st{\kappa}(\theta) : C \xrightarrow{\Delta} C \otimes C \xrightarrow{\theta \otimes \id - \id \otimes \theta} C.
  \end{equation}
\end{example}

\begin{example}
  \label{exa:curved-Lie}
  For Lie coalgebras and the Koszul twisting morphism $\varphi = \kappa : \Com^{\ashk} = (\mathscr{S}^c)^{-1}\Lie^\vee \to \Sigma\Com$, we recover (the dual of) the notion of curved Lie algebras~\cite{ChuangLazarevMannan2016,Maunder2017}, i.e, Lie algebras $\mathfrak{g}$ equipped with a derivation $d$ of degree $-1$ and an element $\omega = \theta^\vee$ of degree $-2$ such that $d^{2} = [\omega, -]$.
\end{example}

We also define the notion of a semi-augmented algebra over $u\PP$ (the terminology is adapted from~\cite{HirshMilles2012}).
This is necessary because, in general, the bar construction of an algebra is not a curved coalgebra.

\begin{definition}
  \label{def:semi-aug}
  A semi-augmented $u\PP$-algebra is a dg-$u\PP$-algebra $A$ equipped with a linear map $\varepsilon_{A} : A \to \Bbbk$ (not necessarily compatible with the dg-algebra structure), such that $\varepsilon_{A}(\uni) = 1$.
  Given such a semi-augmented $u\PP$-algebra, we let $\bar{A}$ be the kernel of $\varepsilon_{A}$.
\end{definition}

For a semi-augmented $u\PP$-algebra $(A, \varepsilon_{A})$, the exact sequence $0 \to \bar{A} \to A \xrightarrow{\varepsilon} \Bbbk\uni \to 0$ defines an isomorphism of graded modules $A \cong \bar{A} \oplus \Bbbk \uni$.
This isomorphism is not compatible with the algebra structure in general.
This allows to define a ``composition'' $\bar\gamma_{A} : u\PP(\bar{A}) \to \bar{A}$ (which is generally not associative) and a differential $\bar{d}_{A} : \bar{A} \to \Sigma \bar{A}$, by using the inclusion and the projection $\bar{A} \to \bar{A} \oplus \Bbbk \uni \cong A \to \bar{A}$.
Note that $\bar{d}_A^2 = 0$, since if $dx = \bar{d}x + \alpha 1$, then $d(\bar{d}x) = d(\bar{d}x + \alpha 1) = d^2x = 0 = 0 + 0 \cdot 1_A$.

\subsection{Cobar construction}
\label{sec:cobar-construction}

Let $\varphi : \CC \to \PP$ be a twisting morphism, i.e., an element satisfying the Maurer--Cartan equation $\varphi \star \varphi = 0$.
Let $C = (C, d_{C}, \theta_{C})$ be a $\varphi$-curved $\mathscr{S}^c\CC$-coalgebra as in Definition~\ref{def:curved-coalg}.
We adapt the definition of~\cite[Section~5.2.5]{HirshMilles2012}.

\begin{definition}\label{def:cobar}
  The cobar construction of $C$ with respect to $\varphi$ is:
  \begin{equation}
    \Omega_{\varphi}(C) \coloneqq (u\PP(\Sigma^{-1} C), d_{\Omega} = -d_{0} - d_{1} - d_{2}),
  \end{equation}
  where each $d_{i}$ is a derivation of degree $-1$ defined on generators by:
  \begin{align}
    d_{0}|_{\Sigma^{-1}C}
     & : \Sigma^{-1}C \xrightarrow{\theta_{C}} \Sigma \Bbbk \uni
    \hookrightarrow \Sigma u\PP(\Sigma^{-1}C)
    \\
    d_{1}|_{\Sigma^{-1}C}
     & : \Sigma^{-1}C \xrightarrow{d_{C}} \Sigma\Sigma^{-1}C \hookrightarrow
    \Sigma u\PP(\Sigma^{-1}C)
    \\
    d_{2}|_{\Sigma^{-1}C}
     & : \Sigma^{-1}C \xrightarrow{\Delta} \CC(\Sigma^{-1}C) \xrightarrow{\varphi(\id)} \Sigma u\PP(\Sigma^{-1}C)
  \end{align}
  It is equipped with the semi-augmentation $\varepsilon_{\Omega} : \Omega_{\varphi}(C) \to \Bbbk$ induced by the projection $u\PP \to \uni$.
\end{definition}

\begin{example}
  Let $u\PP = u\Ass$ be the operad encoding unital associative algebras and $\kappa : \Ass^{\ashk} \to \Sigma\Ass$ the Koszul twisting morphism.
  We saw that $\kappa$-curved $\mathscr{S}^c\Ass^{\ashk}$-coalgebras are curved coassociative coalgebras (Example~\ref{exa:curved-coass}).
  The above cobar construction is the classical curved cobar construction.
\end{example}

\begin{proposition}\label{prop:omega-d-square}
  Given a $\varphi$-curved $\mathscr{S}^c\CC$-coalgebra $(C, d_{C}, \theta_{C})$, the cobar construction $\Omega_{\varphi}(C)$ is a semi-augmented $u\PP$-algebra.
\end{proposition}

\begin{proof}
  All we need to check is that the derivation $d_{\Omega}$ squares to zero.
  There is a weight decomposition (denoted $\omega$) of $u\PP(\Sigma^{-1}C)$ obtained by assigning $\Sigma^{-1}C$ the weight one.
  For example, $d_{0}$ is of weight $-2$, $d_{1}$ is of weight $-1$, and $d_{2}$ is of weight zero.
  We may then decompose $d^{2}_{\Omega}$ in terms of this weight:
  \begin{equation*}
    d_{\Omega}^{2} = \underbrace{d_{0}^{2}}_{-4} + \underbrace{d_{0}d_{1}
    + d_{1}d_{0}}_{-3} +
    \underbrace{d_{1}^{2} + d_{0}d_{2} +
    d_{2}d_{0}}_{-2} +\underbrace{d_{1}d_{2} +
    d_{2}d_{1}}_{-1} + \underbrace{d_{2}^{2}}_{0}.
  \end{equation*}
  Each summand is a derivation (because $d_{\Omega}^{2} = \frac{1}{2} [d_{\Omega}, d_{\Omega}]$ is a derivation thus so are its weight components).
  It thus suffices to check that these summands vanish on generators.
  \begin{itemize}
    \item $d_{0}^{2} = 0$: the image of $d_{0}$ is included in $\Bbbk \uni$, and every derivation vanishes on $\uni$;
    \item $d_{1} d_{0} = d_{0} d_{1} = 0$ respectively because $d_{1}(\uni) = 0$ and $\theta_{C} d_{C} = 0$;
    \item $d_{1}^{2} + d_{0} d_{2} = d_{2} d_{0} = 0$: we have that $d_{2} d_{0} = 0$, again because $d_{2}(\uni) = 0$, and $d_{1}^{2} + d_{0} d_{2} = 0$ follows from $d_{C}^{2} = \st{\varphi}(\theta_{C})$;
    \item $d_{1} d_{2} + d_{2} d_{1} = 0$ comes from the fact that $d_{C}$ is a derivation, that relation being post-composed by $\varphi$ to obtain $d_1d_2+d_2d_1=0$;
    \item $d_{2}^2 = 0$ follows from the Maurer--Cartan equation $\varphi \star \varphi = 0$ and the commutativity of the following diagram (where we do not write all suspensions):
          \[\begin{tikzcd}[column sep = huge]
              C \rar{\Delta_{C}} \dar{\Delta_{C}}
              & \CC(C) \rar{\varphi \circ \id} \dar{\id \circ' \Delta_{C}}
              & u\PP(C) \dar{\id \circ' \Delta_{C}} \\
              \CC(C) \rar{\Delta_{(1)} \circ \id}
              & \CC(C; \CC(C)) \rar{\varphi(\id; \id)}
              & u\PP(C; \CC(C)) \dar{\id(\id; \varphi(\id))} \\
              {} & {} & (u\PP \circ_{(1)} u\PP)(C) \dar{\gamma_{(1)}} \\
              {} & {} & u\PP(C).
            \end{tikzcd}
            \qedhere
          \]
  \end{itemize}
\end{proof}

\subsection{Bar construction}
\label{sec:bar-construction}

We now define the adjoint of the cobar construction: the bar construction (cf.~\cite[Section~3.3.2]{HirshMilles2012} for the (pr)operadic case).
Let $A$ be a semi-augmented dg-$u\PP$-algebra (see Definition~\ref{def:semi-aug}), $\varepsilon_{A} : A \to \Bbbk$ be the semi-augmentation, and $\bar{A} = \ker \varepsilon_{A}$.
Recall that we have induced linear maps $\bar{\gamma}_{A} : u\PP(\bar{A}) \to \bar{A}$ and $\bar{d}_{A} : \bar{A} \to \bar{A}$.

We now define the $\varphi$-curved $\mathscr{S}^c\CC$-coalgebra (cf.~\cite[Section~5.2.3]{HirshMilles2012})
\begin{equation}
  B_{\varphi}A \coloneqq (\Sigma \CC(\bar{A}), d_{B} = d_1 + d_2, \theta_{B}).
\end{equation}
The underlying $\mathscr{S}^c\CC$-coalgebra of $B_{\varphi}A$ is merely the shifted cofree coalgebra $\Sigma \CC(\bar{A})$.
The predifferential $d_{B}$ is the sum of the unique coderivations $d_1,d_2$ whose corestrictions are respectively:
\begin{align}
  d_{2}|^{\Sigma\bar{A}} : \Sigma\CC(\bar{A}) \xrightarrow{\Sigma\varphi(\id)} \Sigma^2 \PP(\bar{A}) \xrightarrow{\bar\gamma_{A}} \Sigma^2 \bar{A};
  \\
  d_{1}|^{\Sigma\bar{A}} : \Sigma\CC(\bar{A}) \twoheadrightarrow \Sigma\bar{A} \xrightarrow{\bar{d}_{A}} \Sigma^2 \bar{A}.
\end{align}

Let $\varepsilon_{\CC} : \CC \to I$ be the counit of the cooperad $\CC$.
The curvature $\theta_{B} : \Sigma\CC(A) \to \Bbbk$ is the map of degree $-2$ given by:
\begin{equation}\label{eq:thetaB}
  \Sigma\CC(\bar{A}) \xrightarrow{(\varepsilon_{\CC} \oplus \varphi)(\id_{A})} \Sigma\bar{A} \oplus \Sigma^2 u\PP(\bar{A}) \xrightarrow{d_{A} + \gamma_{A}} \Sigma^2 A \xrightarrow{\varepsilon_{A}} \Sigma^2 \Bbbk.
\end{equation}
Concretely, let us say that $\Sigma c(a_1, \dots, a_n) \in B_\varphi A$ has weight $n$.
Then:
\begin{itemize}
  \item $\theta_B(\Sigma\id_{\CC}(a)) = \varepsilon_A(d_A a)$ on elements of weight one;
  \item $\theta_B(\Sigma c(a, a')) = \varepsilon_A(\gamma_A(\varphi(c), a, a'))$ on elements of weight two;
  \item $\theta_B$ vanishes on all elements of weight $\geq 3$ since $\im \varphi \subset \PP(2)$.
\end{itemize}

Compare the following with~\cite[Lemma~3.3.3]{HirshMilles2012}.

\begin{proposition}\label{prop:bar-is-curved}
  The data $B_{\varphi}A = (\Sigma\CC(\bar{A}), d_{B}, \theta_{B})$ defines a $\varphi$-curved coalgebra from the semi-augmented $u\PP$-algebra $A$.
\end{proposition}

\begin{proof}
  Let us first check that $\theta_{B}d_{B} = 0$.
  The curvature $\theta_B$ is only nonzero on elements of weight at most two, the summand $d_1$ of the differential preserves the weight, and the summand $d_2$ decreases the weight by exactly one, due to our hypothesis that $\im \varphi \subset \PP(2)$.
  We thus only need to check the equality on elements $x \in B_\varphi A$ of weight $\leq 3$.
  \begin{itemize}
    \item If $x = \Sigma c(a)$ has weight one, then either $c = \id_{\CC}$, in which case $\theta_B(d_B(\Sigma\id_{\CC}(a))) = \varepsilon_A(\bar{d}_A^2 a) = 0$, or $c$ is not a multiple of $\id_{\CC}$, in which case $\theta_B(d_B x) = 0$ by definition.
    \item If $x = \Sigma c(a, a')$ has weight two, then
          \begin{multline}
            d_B x =
            (-1)^{ |c|+1} \Sigma c(\bar{d}_{A} a, a')
            + (-1)^{|c|+|a|+1} \Sigma c(a, \bar{d}_{A} a') \\
            + \id_{\CC}(\bar{\gamma}_A(\varphi(c), a, a')).
          \end{multline}
          and thus $\theta_B d_B x = 0$ simply follows from the compatibility of the $\PP$-structure with the differential.
    \item If $x = \Sigma c(a, a', a'')$, then $\theta_B d_1 x = 0$ for weight reasons.
          Moreover, using the associativity of the $\PP$-structure on $A$, we can compute easily that $\theta_B d_B x = \varepsilon_A(\gamma_A(\varphi \star \varphi, a, a', a''))$, which vanishes from the Maurer--Cartan equation $\varphi \star \varphi = 0$.
  \end{itemize}

  Let us now check that $d_{B}^{2} = \st{\varphi}(\theta_{B})$.
  It is enough to check this when projected on cogenerators, as $d_B^2 = \frac{1}{2} [d_B, d_B]$ and $\st{\varphi}(\theta_B)$ are coderivations (Lemma~\ref{lem:star-coder}).
  Thanks to the explicit description of Definition~\ref{def:phi-theta}, we find that the projection of $\st{\varphi}(\theta_B)(x)$ on cogenerators is nonzero only on elements of weight two and three.
  Moreover, $d_1^2(x) = 0$, the projection of $(d_1d_2+d_2d_1)(x)$ can only be nonzero on elements of weight two, and the projection of $d_2^2(x)$ is only nonzero on elements of weight three.
  If $x = \Sigma c(a, a')$, then:
  \begin{equation}
    \st{\varphi}(\theta_B)(x)|^{\Sigma \bar{A}} = \gamma_{u\PP}(\varphi(c), \varepsilon_A(da) \uni, a') + \gamma_{u\PP}(\varphi(c), a, \varepsilon_A(da')).
  \end{equation}
  This is easily seen to be equal to the projection of $(d_1d_2 + d_2d_1)(x)$ from the computations above.
  A similar computation shows that if $x$ has weight three, then $\st{\varphi}(\theta_B)(x)$ has the same projection on cogenerators as $d_2^2(x)$.
\end{proof}

\begin{example}
  Let $u\PP = u\Ass$, $\varphi = \kappa : \Ass^{\ashk} = (\mathscr{S}^c)^{-1} \Ass^\vee \to \Sigma \Ass$, and let $A$ be a semi-augmented unital associative algebra.
  Then $B_\kappa A$ is the shifted cofree $\Ass^\vee$-coalgebra on $A$, i.e., the cofree coassociative coalgebra on $\Sigma A$.
  This recovers the classical curved bar construction.
\end{example}

\subsection{Adjunction}
\label{sec:adjunction}

\begin{definition}
  Let $\varphi : \CC(2) \to \Sigma\PP(2)$ be an injective twisting morphism, let $(C, d_{C}, \theta_{C})$ be a $\varphi$-curved $\mathscr{S}^c\CC$-coalgebra, and let $A$ be a semi-augmented $u\PP$-algebra.
  The set of $\varphi$-twisting morphisms from $C$ to $A$ is:
  \begin{equation}
    \Tw_{\varphi}(C,A) \coloneqq \{ \beta : \Sigma^{-1}C \to \bar{A} \mid \partial(\beta) + \hst{\varphi}(\beta) + \Theta_C^{A} = 0 \},
  \end{equation}
  where $\partial(\beta) = d_A \beta + (\Sigma^2\beta) d_C$, and $\hst{\varphi}(\beta)$ and $\Theta_C^A$ are given by:
  \begin{align}
    \hst{\varphi}(\beta) & : C \xrightarrow{\Delta_{C}} \Sigma\CC(\Sigma^{-1}C) \xrightarrow{\Sigma \varphi(\beta)} \Sigma^2 u\PP(A) \xrightarrow{\Sigma^2 \gamma_{A}} \Sigma^2 A, \\
    \Theta_C^A           & : C \xrightarrow{\Sigma \theta_{C}} \Sigma^2 \Bbbk \uni \to \Sigma^2 A,
  \end{align}
  and $\Bbbk \uni \to A$ is defined using the action of $\uni \in u\PP(0)$ on $A$.
\end{definition}

We then have the following ``Rosetta Stone'' (cf.~\cite{LodayVallette2012}):
\begin{proposition}
  \label{prop:rosetta}
  Let $C$ be a $\varphi$-curved $\mathscr{S}^c\CC$-coalgebra and $A$ be a semi-augmented $u\PP$-algebra.
  Then there are natural bijections (in particular, $\Omega_{\varphi}$ and $B_{\varphi}$ are adjoint):
  \begin{equation}
    \begin{aligned}
      \Hom_{\text{\textup{sem.aug.}} u\PP \text{\textup{-alg}}}(\Omega_{\varphi}C, A)
       & \cong \Tw_{\varphi}(C,A)                                                                  \\
       & \cong \Hom_{\varphi\text{\textup{-curved} } \CC \text{\textup{-coalg}}}(C, B_{\varphi}A).
    \end{aligned}
  \end{equation}
\end{proposition}

\begin{proof}
  Let us first prove the existence of the first bijection.
  Given $\beta \in \Tw_{\varphi}(C,A)$, we let $f_{\beta} : u\PP(\Sigma^{-1}C) \to A$ be the $u\PP$-algebra morphism given on generators by $\beta$.
  We must check that $f_{\beta} d_{\Omega} = d_{A} f_{\beta}$.
  As we are working with derivations and morphisms, it is enough to check this on the generators $\Sigma^{-1} C$.
  Recall that $d_\Omega = -d_0 -d_1 -d_2$, where the summands are respectively defined using the curvature, the predifferential, and the coalgebra structure of $C$.
  The restrictions of the maps involved are:
  \begin{itemize}
    \item $f_{\beta} d_{0} |_{\Sigma^{-1} C} = \Theta_C^{A}$;
    \item $f_{\beta} d_{1} |_{\Sigma^{-1} C} = (\Sigma^2\beta) d_{C}$;
    \item $f_{\beta} d_{2} |_{\Sigma^{-1} C} = \Sigma^{-1} \hst{\varphi} \beta$ is obtained as the composite $\Sigma^{-1} C \xrightarrow{\Delta_{C}} \CC(\Sigma^{-1}C) \xrightarrow{\varphi(\beta)} \Sigma u\PP(A) \xrightarrow{\Sigma \gamma_{A}} \Sigma A$;
    \item $d_{A} f_{\beta} |_{\Sigma^{-1} C} = d_{A} \beta$.
  \end{itemize}
  We can thus compute that:
  \begin{equation*}
    \begin{aligned}
      (d_A f_\beta - f_\beta d_\Omega) |_{\Sigma^{-1} C}
       & = (d_A f + f d_0 + f d_1 + f d_2)|_{\Sigma^{-1} C}            \\
       & = d_A \beta + \Theta_C^A + (\Sigma^2\beta) d_C + \hat{\star}_\varphi \beta \\
       & = \Theta_C^A + \partial(\beta) + \hat{\star}_\varphi \beta \\
       & = 0 \text{ by the Maurer--Cartan equation}.
    \end{aligned}
  \end{equation*}

  Conversely, given $f : \Omega_{\varphi}C \to A$, then we can define $\beta \coloneqq f|_{\Sigma^{-1} C}$.
  The same proof as above but in the reverse direction shows that the compatibility of $f$ with the differentials implies the Maurer--Cartan equation.
  Moreover, the two constructions are inverse to each other.

  The definition of the second bijection is similar (see also the proof of~\cite[Theorem~3.4.1]{HirshMilles2012} for the case of (co)operads).
  Given a twisting morphism $\beta \in \Tw_\varphi(C,A)$, the morphism $g_\beta : C \to B_\varphi A$ is defined as the unique morphism of $\mathscr{S}^c\CC$-coalgebras $C \to \Sigma \CC(\bar{A})$ with corestriction $\beta$.
  The fact that $g_\beta$ commutes with the predifferentials and the curvatures of $C$ and $B_\varphi A$ follows from the Maurer--Cartan equation, as a similar argument shows.

  All that remains is checking that the two bijections are natural in terms of $C$ and $A$.
  This is a simple exercise in commutative diagrams.
\end{proof}

\section{Koszul duality of unitary algebras}
\label{sec:kosz-unit-algebra}

\subsection{Algebras with quadratic-linear-constant relations}
\label{sec:algebr-with-quadr}

We now define the type of algebras for which we will develop a Koszul duality theory, namely algebras with quadratic-linear-constant (QLC) relations.
For this we adapt the notion of a monogenic algebra of~\cite[Section~4.1]{Milles2012} (see Section~\ref{sec:kosz-dual-quadr}).
We still assume that we are given a unital version $u\PP = \Free(\uni \oplus E) / (R + R')$ of a binary quadratic operad $\PP = \Free(E) / (R)$ as in Section~\ref{sec:unital-versions}.

\begin{definition}
  \label{def:alg-qlc}
  An $u\PP$-algebra with QLC relations is a $u\PP$-algebra $A$ with $d_{A} = 0$, equipped with a presentation by generators and relations:
  \begin{equation}
    A = u\PP(V) / I,
  \end{equation}
  satisfying the two conditions:
  \begin{enumerate}
    \item the ideal $I$ is generated by $S \coloneqq I \cap (\uni \oplus V \oplus E(V))$ (where $E(V) = E \otimes_{\Sigma_{2}} V^{\otimes 2}$),
    \item the relations in $S$ all contain quadratic terms, $S \cap (\uni \oplus V) = 0$.
  \end{enumerate}
\end{definition}

\begin{remark}
  Definition~\ref{def:alg-qlc} implies that $S$ is the graph of some map
  \begin{equation}
    \label{eq:alpha}
    \alpha = (\alpha_{0} \oplus \alpha_{1}) : qS \to \Bbbk \uni \oplus V,
  \end{equation}
  i.e., $S = \{ x + \alpha(x) \mid x \in qS\}$.
  Moreover, such an algebra is automatically semi-augmented (Definition~\ref{def:semi-aug}) as we are working over a field $\Bbbk$.
\end{remark}

\begin{definition}
  \label{def:quad-red}
  Given a $u\PP$-algebra $A$ with QLC relations as above, let $qS$ be the projection of $S$ onto $E(V)$.
  Then the quadratic reduction $qA$ of $A$ is the monogenic $\PP$-algebra obtained by:
  \begin{equation}
    qA \coloneqq \PP(V) / (qS).
  \end{equation}
\end{definition}

Until the end of this section, $A$ will be a $u\PP$-algebra with QLC relations, with the same notation as in this subsection.

\subsection{Koszul dual coalgebra}
\label{sec:kosz-dual-coalg}

Let $\PP^{\ashk}$ be the Koszul dual cooperad of $\PP$, with an operad-twisting morphism $\kappa : \PP^{\ashk} \to \PP$ (see Section~\ref{sec:quadratic-operads}).
The theory of~\cite{Milles2012} (see Section~\ref{sec:kosz-dual-quadr}) defines a Koszul dual $\mathscr{S}^{c}\PP^{\ashk}$-coalgebra $qA^{\ashk}$ from the quadratic reduction $qA$:
\begin{equation}
  qA^{\ashk} \coloneqq \Sigma \PP^{\ashk}(V, \Sigma qS).
\end{equation}

Using the map $\alpha$ from Equation~\eqref{eq:alpha}, we may define $d_{A^{\ashk}} : qA^{\ashk} \to \Sigma qA^{\ashk}$ to be the unique coderivation whose corestriction is given by:
\begin{equation}
  d_{A^{\ashk}}|^{\Sigma V} : qA^{\ashk} \twoheadrightarrow \Sigma(\Sigma qS) \xrightarrow{\Sigma^2 \alpha_{1}} \Sigma(\Sigma V) \subset \Sigma qA^{\ashk}.
\end{equation}

Moreover, we define a map $\theta_{A^{\ashk}}$ of degree $-2$ by:
\begin{equation}
  \theta_{A^{\ashk}} : \Sigma^{-1}qA^{\ashk} \twoheadrightarrow \Sigma qS \xrightarrow{\Sigma \alpha_{0}} \Sigma \Bbbk \uni.
\end{equation}
As before, for $x \in qA^{\ashk}$, we let $\theta_{A^{\ashk}}(\Sigma^{-1}x) = \Theta(x) \cdot \Sigma \uni$ where $\Theta(x) \in \K$.

We now define the Koszul dual coalgebra of $A$ by adapting~\cite[Section~4.2]{HirshMilles2012}.
The proof of the following proposition is heavily inspired from that of~\cite[Lemma~4.1.1]{HirshMilles2012}.

\begin{proposition}
  \label{prop:def-koszul-dual}
  The following data defines a $\kappa$-curved $\mathscr{S}^c\PP^{\ashk}$-coalgebra, called the Koszul dual curved coalgebra of $A$:
  \begin{equation}
    A^{\ashk} \coloneqq (qA^{\ashk}, d_{A^{\ashk}}, \theta_{A^{\ashk}}).
  \end{equation}
\end{proposition}

\begin{proof}
  We must show that $d_{A^{\ashk}}(qA^{\ashk}) \subset \Sigma qA^{\ashk}$, that $\st{\kappa}(\theta_{A^{\ashk}}) = d_{A^{\ashk}}^{2}$, and that $\theta_{A^{\ashk}} d_{A^{\ashk}} = 0$.
  Just like in the proof of Proposition~\ref{prop:bar-is-curved}, we can decompose $qA^{\ashk} = \Sigma \PP^{\ashk}(V, \Sigma qS)$ by weight, i.e., the number $n$ of generators $v_i \in V$ in an expression of the form $\Sigma x(v_1, \dots, v_n)$.
  The differential $d_{A^{\ashk}}$ decreases the weight by exactly one, while $\theta_{A^{\ashk}}$ is only nonzero on elements of weight two.
  Using considerations similar to the proof of Proposition~\ref{prop:bar-is-curved}, we only need to check the three equalities above on elements of $qA^{\ashk}$ of weight three.

  Let thus $Y \in (qA^{\ashk})^{(3)}$ be an element of weight three.
  Its image under the $\mathscr{S}^c\PP^{\ashk}$-structure map of $qA^{\ashk}$ is the element $\Delta(Y) \in \Sigma\PP^{\ashk}(\Sigma^{-1}qA^{\ashk})$.
  By definition of $qA^{\ashk}$, we have the following:
  \begin{equation}
    \Delta(Y) \in \bigl( \Sigma^2 E \otimes (V \otimes \Sigma qS) \bigr) \cap \bigl( \Sigma^2 E \otimes (\Sigma qS \otimes V) \bigr) \subset \Sigma\PP^{\ashk}(\Sigma^{-1}qA^{\ashk}).
  \end{equation}

  In other words, we have two decompositions
  \begin{equation}
    \Delta(Y) = \sum_{i} (\Sigma^2 \rho_{i})\bigl(v_{i}, \Sigma X_{i}\bigr) = \sum_{j} (\Sigma^2 \rho'_{j})\bigl(\Sigma X'_{j}, v'_{j}\bigr),
  \end{equation}
  where $\rho_{i}, \rho'_{j} \in E$, $v_{i}, v'_{j} \in V$, and $X_{i}, X'_{j} \in qS$.
  Then we get, using the fact that $d_{A^{\ashk}}$ is a coderivation and its projection on cogenerators:
  \begin{equation}
    d_{A^{\ashk}}(Y)
      = -\sum_{i} (\Sigma^2\rho_{i})\bigl(v_{i}, \Sigma\alpha_{1}(X_{i})\bigr) + \sum_{j} (\Sigma^2\rho'_{j})\bigl(\Sigma \alpha_{1}(X'_{j}), v'_{j}\bigr)
      \in qA^{\ashk}.
  \end{equation}
  And similarly, we have:
  \begin{multline}
    \st{\kappa}(\theta_{A^{\ashk}})(Y) = - \sum_{i} \Sigma^2 \bigl(\rho_{i} \circ_2 \alpha_0(X_i)\bigr) \cdot \Sigma v_{i} + \sum_{j} \Sigma^2 \bigl(\rho'_{j} \circ_1 \alpha_0(X'_j)\bigr) \cdot \Sigma v'_{j} \\
    \in \Sigma^2(\Sigma V) \subset \Sigma^{2}qA^{\ashk},
  \end{multline}
  where $\rho_i \circ_2 \alpha_0(X_i), \rho'_j \circ_1 \alpha_0(X'_j) \in u\PP(1) \cong \K \id$ are scalars.
  We know that $X + \alpha_1(X) + \alpha_0(X) \in S$ for any element $X \in qS$.
  If we add and subtract the extra term $\Delta(Y)$ from the sum of the previous equations, we obtain an element of the ideal generated by $S$, i.e., we have (where we use the condition of Definition~\ref{def:alg-qlc}):
  \begin{equation}
    \label{eq:complicated}
    d_{A^{\ashk}}(Y) + \st{\kappa}(\theta_{A^{\ashk}})(Y)
    \in \Sigma^2 \bigl((S) \cap (\uni \oplus V \oplus E(V))\bigr) = \Sigma^{2} S.
  \end{equation}
  Since $S$ is the graph of $\alpha$, the expression of Equation~\eqref{eq:complicated} is of the form $\Sigma^2\bigl(x + \alpha_{1}(x) + \alpha_{0}(x)\bigr)$ for some $x \in qS$.
  The element $d_{A^{\ashk}}(Y)$ is of weight two, while $\st{\kappa}(\theta_{A^{\ashk}})(Y)$ is of weight one.
  Thus we must have that $x$ is of weight two, and so by identifying each weight component we obtain that:
  \begin{itemize}
    \item the element $d_{A^{\ashk}}(Y) = \Sigma^2 x$ belongs to $\Sigma^{2} qS = (qA^{\ashk})^{(2)}$;
    \item the element $d_{A^{\ashk}}^{2}(Y)$ is equal to $(\Sigma^2 \alpha_{1})(d_{A^{\ashk}}(Y))$, which we know from Equation~\eqref{eq:complicated} is equal to $\st{\kappa}(\theta_{A^{\ashk}})(Y) = \Sigma^2 \alpha_1(x)$ (i.e., the weight one part);
    \item similarly, $\theta_{A^{\ashk}} d_{A^{\ashk}}(Y)$ is equal to $\Sigma^2\alpha_{0}(d_{A^{\ashk}}(Y))$, which vanishes because there is no element of weight zero in Equation~\eqref{eq:complicated}.
    \qedhere
  \end{itemize}
\end{proof}

\subsection{Main theorem}
\label{sec:main-theorem}

Let us now define $\varkappa : \Sigma^{-1}qA^{\ashk} \to \bar{A}$ of degree $-1$ by:
\begin{equation}
  \varkappa : \Sigma^{-1}qA^{\ashk} \twoheadrightarrow V \hookrightarrow \bar{A}.
\end{equation}

\begin{proposition}
  The morphism $\varkappa$ satisfies the curved Maurer--Cartan equation, i.e., it is an element of $\Tw_{\varphi}(qA^{\ashk}, A)$ from Proposition~\ref{prop:rosetta}:
  \begin{equation}
    \partial(\varkappa) + \hst{\kappa}(\varkappa) + \Theta_C^{A} = 0.
  \end{equation}
\end{proposition}

\begin{proof}
  We can rewrite the above equation as:
  \begin{equation}
    (\Sigma^2\varkappa) d_{A^{\ashk}} + \gamma_{A} \circ (\kappa(\varkappa)) \circ \Delta_{qA^{\ashk}} + \Theta_C^{A} = 0.
  \end{equation}
  Moreover, by checking the definitions, we see that:
  \begin{itemize}
    \item $(\Sigma^2\varkappa) d_{A^{\ashk}}$ is obtained as $qA^{\ashk} \twoheadrightarrow \Sigma^{2} qS \xrightarrow{\Sigma^2\alpha_{1}} \Sigma^2 V \hookrightarrow \Sigma^2 A$;
    \item $\Theta_C^{A}$ is obtained as $qA^{\ashk} \twoheadrightarrow \Sigma^{2} qS \xrightarrow{\Sigma^2\alpha_{0}} \Sigma^2 \Bbbk \uni \hookrightarrow \Sigma^2 A$;
    \item $\hst{\kappa}(\varkappa)$ vanishes everywhere except on $(qA^{\ashk})^{(2)} = \Sigma^{2} qS$, where it is equal to $\gamma_{A} \circ \Delta_{qA^{\ashk}}$.
  \end{itemize}

  The image of $\hst{\kappa}(\varkappa) + \varkappa d_{A^{\ashk}} + \Theta_C^{A}$ is thus included in the image of the graph of $\alpha$ under $\gamma_{A}$.
  But this graph is $S$, and $\gamma_{A}(S) = 0$ in $A = u\PP(V)/(S)$.
\end{proof}

\begin{definition}
  The $u\PP$-algebra $A$ is said to be Koszul if the $\PP$-algebra $qA$ is Koszul in the sense of~\cite{Milles2012}.
\end{definition}

Using the Rosetta Stone (Proposition~\ref{prop:rosetta}), $\varkappa$ defines a morphism $f_{\varkappa} : \Omega_{\kappa}A^{\ashk} \to A$.
Recall that $qA$ is Koszul if and only if the induced morphism $\Omega_{\kappa}(qA^{\ashk}) \to qA$ is a quasi-isomorphisms~\cite[Theorem~4.9]{Milles2012}.
Our definition is justified by the following theorem:
\begin{theorem}
  \label{thm:main}
  If $A$ is Koszul, then $f_{\varkappa} : \Omega_{\kappa}A^{\ashk} \to A$ is a cofibrant resolution of $A$ in the semi-model category of $u\PP$-algebras defined in~\cite[Theorem~12.3.A]{Fresse2009}.

\end{theorem}

\begin{proof}
  Let us filter $A$ and $\Omega_{\kappa}A^{\ashk}$ by the weight in terms of $V$.
  It is clear that $f_{\varkappa}$ is compatible with this filtration.
  The summands $d_{0}$ and $d_{1}$ of $d_{\Omega}$ strictly decrease this filtration, while $d_{2}$ preserves it.
  Thus, on the first pages of the associated spectral sequences, we obtain the morphism:
  \begin{equation}
    \Omega_{\kappa}(qA^{\ashk}) \oplus \Bbbk \uni \to qA \oplus \Bbbk \uni.
  \end{equation}

  Our hypothesis on $qA$ and~\cite[Theorem~4.9]{Milles2012} implies that this is a quasi-isomorphism, i.e., we have an isomorphism on the second pages of the spectral sequences.
  The filtration is exhaustive and bounded below, therefore $f_{\varkappa}$ itself is a quasi-isomorphism (see e.g.~\cite[Theorem~3.5]{McCleary2001}).

  It remains to check that $\Omega_{\kappa}A^{\ashk}$ is cofibrant in the semi-model category from~\cite[Theorem~12.3.A]{Fresse2009}, which applies as we are working over a field of characteristic zero and so $u\PP$ is always $\Sigma_{*}$-cofibrant.
  The cobar construction is quasi-free, i.e., free as an algebra if we forget the differential.
  It is moreover equipped with a filtration where $\Sigma x(v_1, \dots, v_n) \in A^{\ashk}$ (for $x \in \PP^{\ashk}$ and $v_i \in V$) is in filtration level $n$.
  Let us check that this filtration satisfies the hypotheses of~\cite[Proposition~12.3.8]{Fresse2009}.
  The summands $d_0$ and $d_1$ of the differential decrease $n$.
  Since $\PP$ is binary, $\im \kappa \subset \PP^{\ashk}(2)$, so $d_2$ decomposes an element of weight $n$ as a product of elements of weights $n_1+n_2=n$ and so $n_1, n_2 < n$ (as there is no arity zero generator in $\PP^{\ashk}$).
  It follows that $\Omega_{\kappa}A^{\ashk}$ is cofibrant.
\end{proof}

\section{Application: symplectic Poisson \texorpdfstring{$n$}{n}-algebras}
\label{sec:appl-sympl-n}

\subsection{Definition}
\label{sec:definition}

In this section, with deal with Poisson $n$-algebras for some integer $n$.
We will abbreviate as $\Lie_{n}$ the operad of Lie algebras shifted by $n-1$, i.e., $\Lie_{n} \coloneqq \mathscr{S}^{1-n}\Lie$.
Recall from Example~\ref{exa:QLC-operad} the operad $u\Pois_n \cong \Com \circ \Lie_n$, generated by two binary operations $\mu$ (product) and $\lambda$ (bracket) and a unary operation $\uni$.

\begin{definition}
  \label{def:sympl-poisson}
  The $D$th symplectic Poisson $n$-algebra is defined by:
  \begin{equation}
    A_{n;D} \coloneqq \bigl( \Bbbk[x_{1}, \dots, x_{D}, \xi_{1}, \dots, \xi_{D}], \{,\} \bigr).
  \end{equation}
  where the generators $x_{i}$ have degree zero and the $\xi_{i}$ have degree $1-n$.
  The algebra $A_{n;D}$ is free as a unital commutative algebra, and the Lie bracket is defined on generators by:
  \begin{align}
    \{ x_{i}, x_{j} \} & = 0 & \{ \xi_{i}, \xi_{j} \} & = 0 & \{ x_{i}, \xi_{j}\} & = \delta_{ij}.
  \end{align}
\end{definition}

The algebra $A_{n;D} = u\Pois_{n}(V_{n;D}) / (S_{n;D})$ is equipped with a QLC presentation.
The space of generators is $V_{n;D} \coloneqq \R \langle x_{i}, \xi_{j} \rangle$, a graded vector space of dimension $2D$.
We check that the ideal of relations $I_{n;D}$ is generated by the set $S_{n;D}$ given by the three sets of relations fixing the Lie brackets of the generators, namely
\begin{equation}
  S_{n;D} = \R \langle \{x_{i}, x_{j}\}, \{\xi_{i}, \xi_{j}\}, \{x_{i},\xi_{j}\} - \delta_{ij} \uni \rangle.
\end{equation}

\begin{remark}
  We may view $A_{n;D}$ as the Poisson $n$-algebra of polynomial functions on the standard shifted symplectic space $T^{*}\R^{D}[1-n]$.
  The element $x_{i}$ is a polynomial function on the coordinate space $\R^{D}$, and the element $\xi_{j}$ can be viewed as the vector field $\partial / \partial x_{j}$, which is a function on $T^{*}\R^{D}[1-n]$.
\end{remark}

We will fix $n$ and $D$ and drop them from the notation $A = A_{n;D}$ in what follows.

\subsection{Koszul property and explicit resolution}
\label{sec:kosz-prop-expl}

The goal of this section is to prove:

\begin{proposition}
  \label{prop:a-koszul}
  The $u\Pois_{n}$-algebra $A$ is Koszul.
\end{proposition}

\begin{lemma}
  The quadratic reduction $qA$ of $A$ is a free symmetric algebra with trivial Lie bracket.
\end{lemma}

\begin{proof}
  Let $V = \R \langle x_{1}, \dots, x_{D}, \xi_{1}, \dots, \xi_{D} \rangle$ be the generators of $A$.
  We check that $qS = \lambda(V) = \lambda \otimes_{\Sigma_{2}} V^{\otimes 2}$, i.e., in the quadratic reduction, all Lie brackets vanish.
  Therefore, $qA = \Pois_{n}(V) / (qS) = \Pois_{n}(V) / (\lambda(V)) = \Com(V)$ is a free symmetric algebra, and the Lie bracket vanishes.
\end{proof}

Let $\Com^{c}$ be the cooperad governing cocommutative coalgebras, which is the Koszul dual of the Lie operad up to suspension.
Since we are working over a field of characteristic zero, we can identify $\Com^{c}(X)$ with:
\begin{equation}
  \bar{S}^{c}(X) \coloneqq \bigoplus_{i \ge 1} (X^{\otimes i})_{\Sigma_{i}}
\end{equation}
where the coproduct is given by shuffles.
For a shorter notation we will also write $L(X)$ for the free Lie algebra on $X$, $S(X)$ for the free unital symmetric algebra, and $\bar{S}(X)$ for the free symmetric algebra without unit.

\begin{lemma}\label{lem:dual-qa}
  The Koszul dual coalgebra of $qA$ is given by:
  \begin{equation}
    qA^{\ashk} = \Sigma^{1-n} \bar{S}^{c}(\Sigma^{n} V).
  \end{equation}
\end{lemma}

\begin{proof}
  Recall that if $\PP = \QQ_{1} \circ \QQ_{2}$ is obtained by a distributive law between two finitely-generated binary quadratic operads $\QQ_{i} = \Free(E_{i})/(R_{i})$ ($i = 1,2$), then $\PP^! = \QQ_2^! \circ \QQ_1^!$ with the transpose distributive law~\cite[Proposition 8.6.7]{LodayVallette2012}.
  If $A = \QQ_1(V) = \PP(V) / (E_2(V))$, it follows that $A^! = \PP^!(V^*) / (E_2(V)^\perp) = (\QQ_2^! \circ \QQ_1^!)(V^*) / (E_1^\vee(V^\vee)) \cong \QQ_2^!(V^\vee)$ is simply given by the free $\QQ_2^!$-algebra with a trivial action of the generators of $\QQ_1$.
  By dualizing this statement, we thus find that $A^{\ashk} = \Sigma \QQ_{2}^{\ashk}(V)$.

  Applied to our case, we obtain that the Koszul dual coalgebra of $qA = \Com(V)$ is $qA^{\ashk} = \Sigma(\Lie_{n})^{\ashk}(V)$.
  Thanks to the Koszul duality between $\Com$ and $\Lie$, this is identified with $\Sigma^{1-n} \Com^{c}(\Sigma^{n} V) = \Sigma^{1-n}\bar{S}^c(\Sigma^n V)$.
\end{proof}

\begin{proof}[Proof of Proposition~\ref{prop:a-koszul}]
  Let $\kappa : \Pois_{n}^{\ashk} \to \Pois_{n}$ be the twisting morphism of Koszul duality.
  Then the cobar construction of $qA^{\ashk}$ is given by:
  \begin{equation}
    \Omega_{\kappa} qA^{\ashk} = \bigl( \underbrace{\bar{S}(\Sigma^{1-n} L(\Sigma^{n-1}}_{= \Pois_{n}} \Sigma^{-1} \underbrace{\Sigma^{1-n}\bar{S}^{c}(\Sigma^{n} V))}_{= qA^{\ashk}}), d_{2} \bigr).
  \end{equation}
  Here $d_{2}$ is the derivation of $\Pois_{n}$-algebras whose restriction on $\Sigma^{-n} \bar{S}^{c}(\Sigma^{n} V)$ is given by (forgetting about suspension for ease of notation):
  \begin{equation}
    d_{2}|_{\Sigma^{-1}qA^{\ashk}}(u) = \sum_{(u)} \frac{1}{2} [u_{(1)}, u_{(2)}],
  \end{equation}
  where the bracket is the bracket of the free Lie algebra appearing in $\Omega_\kappa qA^\ashk$.
  This $1/2$-factor is due to the identification $\Com^{c}(X) \cong \bar{S}^{c}(X)$, the first being defined using invariants and the second using coinvariants (under the symmetric groups actions).

  The twisting morphism $\varkappa \in \Tw_{\kappa}(qA^{\ashk}, qA)$ is given by $\varkappa(\Sigma x_{i}) = x_{i}$ and $\varkappa(\Sigma \xi_{i}) = \xi_{i}$ on terms of weight one, and it vanishes on terms of weight $\ge 2$.
  This twisting induces a morphism $\Omega_{\kappa}(qA^{\ashk}) \to qA$.
  We easily see that this morphism is the image under $S$ of the bar-cobar resolution of the abelian $\Lie_{n}$ algebra $V$:
  \begin{equation}
    \bigl( \Sigma^{1-n} L(\Sigma^{-1} \bar{S}^{c}(\Sigma^{n} V)), d_{2} \bigr) \xrightarrow{\sim} V,
  \end{equation}
  which is indeed a quasi-isomorphism thanks to the Koszul property of $\Lie_n$.
  The functor $S$ preserves quasi-isomorphisms as we are working over a field of characteristic zero.
  Therefore we obtain that $\Omega_{\kappa}(qA^{\ashk}) \to qA$ is a quasi-isomorphism, thus $qA$ is Koszul, and therefore by definition $A$ is Koszul.
\end{proof}

We then obtain a small resolution of the $u\Pois_{n}$-algebra $A$ using the cobar construction of its Koszul dual coalgebra that we now describe.
The map $\alpha : qS \to \Bbbk \uni \oplus V$ from Equation~\eqref{eq:alpha} is given by $\alpha_{1} = 0$, $\alpha_{0}(\{x_{i}, \xi_{i}\}) = - \uni$ for all $i$, and $\alpha_{0} = 0$ on all other basis elements.
Let us write as a shorthand:
\begin{equation}
  v_{1} \vee \dots \vee v_{k} \coloneqq \frac{1}{k!} \sum_{\sigma \in \Sigma_{k}} v_{\sigma(1)} \otimes \dots \otimes v_{\sigma(k)} \in \Sigma^{1-n} \bar{S}^{c}(\Sigma^{n} V).
\end{equation}
Then the Koszul dual $A^{\ashk} = (qA^{\ashk}, d_{A^{\ashk}}, \theta_{A^{\ashk}})$ is such that $d_{A^{\ashk}} = 0$, and
\begin{align*}
  \theta  : \Sigma^{1-n} \bar{S}^{c}(\Sigma^{n} V) & \to \Bbbk \uni                           \\
  x_{i} \vee \xi_{i}                               & \mapsto - \uni, \quad \text{ for all } i \\
  \text{other basis elements}                      & \mapsto 0.
\end{align*}

We then obtain
\begin{equation}
  \label{eq:cobar-a}
  \Omega_{\kappa} A^{\ashk} = (S(\Sigma^{1-n}L(\Sigma^{-1} \bar{S}^{c}(\Sigma^{n} V))), d_{0} + d_{2}) \xrightarrow{\sim} A,
\end{equation}
where (by abuse of notation) we still denote by $d_{2}$ the Chevalley--Eilenberg differential from before, satisfying $d_{2}(\uni) = 0$.
The derivation $d_{0}$ is the one whose restriction to $\Sigma^{-1} A^{\ashk}$ is given by $\Sigma \theta$:
\begin{equation}
  d_{0}|_{\Sigma^{-1} qA^{\ashk}} : \Sigma^{-n} \bar{S}^{c}(\Sigma^{n} V) \xrightarrow{\Sigma \theta} \Bbbk \uni \hookrightarrow u\Pois_{n}(\Sigma^{-1} qA^{\ashk}).
\end{equation}

\begin{remark}
  \label{rmk:big}
  Let us compare $\Omega_{\kappa}A^{\ashk}$ to the resolution that would be obtained if one applied curved Koszul duality at the level of operads (see Section~\ref{sec:curv-kosz-dual}) with the bar/cobar resolution from the theory of~\cite{HirshMilles2012}.
  Without the suspensions, our resolution is just $SL\bar{S}^{c}(V)$, i.e., it is the free symmetric algebra on the free Lie algebra on the cofree symmetric coalgebra on $V$.

  The resolution that would be obtained from~\cite{HirshMilles2012} would be much bigger (although it can be made explicit).
  Indeed, the quadratic reduction of $u\Pois_{n}$ is not just $\Pois_{n}$, it is in fact the direct sum $\Pois_{n} \oplus \uni$.
  It follows from~\cite[Proposition~6.1.4]{HirshMilles2012}) that $(qu\Pois_{n})^{\ashk}(r)$ is spanned by elements of the type $\bar{\alpha}_{S}$, where $\alpha \in \Pois_{n}(k)$, $S \subset \{ 1, \dots, k\}$ and $r = k - \#S$.
  Roughly speaking, $S$ represents inputs of $\alpha$ that have been ``plugged'' by the counit $\uni$.
  The bar construction of $A$ is then the cofree $(qu\Pois_{n})^{\ashk}$-coalgebra on $A = S(V)$ (plus some differential and we forget about suspensions).
  Then the bar-cobar resolution of $A$ is the free $u\Pois_{n}$-algebra on this bar construction.
  It contains as a subspace $SL\bar{S}^{c}L^{c}(A) = SL\bar{S}^{c}L^{c}S(V)$, which is already quite bigger than $\Omega_{\kappa}A^{\ashk}$, and then we also need to add all operations where inputs have been plugged in by the unit.

  The difference can roughly speaking be explained as follows.
  The bar-cobar resolution of~\cite{HirshMilles2012} knows nothing about the specifics of the algebra $A$, thus it must resolve everything in $A$: the Lie bracket, the symmetric product, and the relations involving the unit.
  This has the advantage of being a general procedure that is independent of $A$ (and functorial).
  But with our specific $A$, we may find a smaller resolution: we know that the product of $A$ has no relations, and the unit is not involved in nontrivial relations (a consequence of the QLC condition), hence they do not need to be resolved.
\end{remark}

\subsection{Derived enveloping algebras}
\label{sec:enveloping-algebras}

\subsubsection{General constructions}
\label{sec:general-construction}

Given an operad $\PP$ and a $\PP$-algebra $A$, the enveloping algebra $\U_{\PP}(A)$ is a unital associative algebra such that the left modules of $\U_{\PP}(A)$ are precisely the operadic left modules of $A$ (see e.g.~\cite[Section~4.3]{Fresse2009}).
Let $\PP[1]$ be the operadic right $\PP$-module given by $\PP[1] = \{ \PP(n+1) \}_{n \ge 0}$.
Then the enveloping algebra $\U_\PP(A)$ can be obtained as the relative composition product:
\begin{equation}
  \U_{\PP}(A) \cong \PP[1] \circ_{\PP} A = \operatorname{coeq} \bigl( \PP[1] \circ \PP \circ A \rightrightarrows \PP[1] \circ A \bigr).
\end{equation}

\begin{example}
  \label{exa:univ-lie}
  The enveloping algebra $\U_{\Lie}(\mathfrak{g})$ of a Lie algebra $\mathfrak{g}$ is the usual universal enveloping algebra $\U(\mathfrak{g})$ of $\mathfrak{g}$.
  We view it as a free associative algebra on symbols $X_{f}$, for $f \in \mathfrak{g}$, subject to the relations $X_{[f,g]} = X_{f} X_{g} - (-1)^{|g| \cdot |f|} X_{g} X_{f}$.
  The universal enveloping algebra $\U_{c\Lie}(\mathfrak{g})$ of a Lie algebra equipped with a central $\uni \in \mathfrak{g}$ is the quotient $\U(\mathfrak{g}) / (X_{\uni})$.
\end{example}

\begin{example}
  \label{exa:univ-com}
  The enveloping algebra $\U_{\Com}(B)$ of a commutative algebra $B$ is $B_{+} = \Bbbk 1 \oplus B$, where $1$ is an extra unit.
  The enveloping algebra $\U_{u\Com}(B)$ of a unital commutative algebra $B$ is $B$ itself (strictly speaking, the quotient of $\Bbbk 1 \oplus B$ by the relation $1 - 1_{B}$).
\end{example}

Suppose now that $\PP$ is any operad (potentially unital, e.g., we could take $\PP = u\Pois_{n}$) and let $A$ be a $\PP$-algebra.
Let $\PP_{\infty} \xrightarrow{\sim} \PP$ be a cofibrant resolution of $\PP$.
The given morphism $\PP_{\infty} \xrightarrow{\sim} \PP$ induces a Quillen equivalence between the semi-model categories of $\PP$- and $\PP_{\infty}$-algebras.
The right adjoint allows us to view $A$ as a $\PP_{\infty}$-algebra.

\begin{proposition}
  \label{prop:derived-u}
  Let $R_{\infty} \xrightarrow{\sim} A$ be a cofibrant resolution of $A$ as a $\PP_{\infty}$-algebra, and let $R \coloneqq \PP \circ_{\PP_{\infty}} R_{\infty}$.
  Then there is an equivalence
  \begin{equation}
    \U_{\PP_{\infty}}(A) \simeq \U_{\PP}(R).
  \end{equation}
\end{proposition}

\begin{proof}
  The proposition follows from the following diagram:
  \begin{equation}
    \begin{tikzcd}
      \U_{\PP_{\infty}}(A) \cong \PP_{\infty}[1]
      \circ_{\PP_{\infty}} A \ar[r, leftarrow, "\sim"] &
      \PP_{\infty}[1] \circ_{\PP_{\infty}} R_{\infty} \ar[d,
        "\sim"] \\
      \PP[1] \circ_{\PP} R \cong \U(R) & \PP[1] \circ_{\PP_{\infty}} R_{\infty} \ar[l, "\cong" swap]
    \end{tikzcd}
    .
  \end{equation}
  The upper horizontal equivalence follows from~\cite[Theorem~17.4.B(b)]{Fresse2009}, and the right vertical one follows from~\cite[Theorem~17.4.A(a)]{Fresse2009}.
  Finally, the bottom horizontal isomorphism follows from the cancellation rule $X \circ_{\PP} (\PP \circ_{\QQ} Y) \cong X \circ_{\QQ} Y$~\cite[Theorem~7.2.2]{Fresse2009}.
\end{proof}

\subsubsection{Poisson case}
\label{sec:poisson-case}

We now consider the symplectic $u\Pois_{n}$-algebra $A = (\R[x_{i}, \xi_{j}], \{\})$ from Definition~\ref{def:sympl-poisson}.
We have the following proposition.

\begin{proposition}\label{prop:cobar-infinity}
  The derived enveloping algebra $\U_{(u\Pois_{n})_{\infty}}(A)$ is quasi-isomorphic to $\U(\Omega_{\kappa}A^{\ashk})$, where $\Omega_{\kappa}A^{\ashk}$ is the cobar construction described in Section~\ref{sec:kosz-prop-expl}.
\end{proposition}

\begin{proof}
  This follows from Proposition~\ref{prop:derived-u}.
  In order to apply that proposition, we need to resolve $A$ as a $(u\Pois_n)_{\infty}$-algebra and push forward that resolution to a $u\Pois_n$-algebra.
  In Definition~\ref{def:cobar}, instead of taking the free $u\PP$-algebra, we can take the free $(u\PP)_{\infty}$-algebra to define the cobar construction and obtain a resolution $R_\infty \to A$ which is cofibrant as a $(u\Pois_n)_\infty$-algebra (using the same result as the end of the proof of Theorem~\ref{thm:main}).
  Since the differentials in the cobar construction only involve generating operations from $u\Pois_n$, we find that $R = u\Pois_n \circ_{(u\Pois_n)_\infty} R_\infty$ is isomorphic to $\Omega_\kappa A^{\ashk}$.
\end{proof}

Both $A$ and $\Omega_{\kappa}A^{\ashk}$ are obtained by considering the relative composition product
\begin{equation}
  S(\Sigma^{1-n} \mathfrak{g}) \coloneqq u\Pois_{n} \circ_{c\Lie_{n}} \Sigma^{1-n} \mathfrak{g},
\end{equation}
where $\mathfrak{g}$ is some $c\Lie$-algebra, and we consider the embedding $c\Lie_{n} \hookrightarrow u\Pois_{n}$.
In other words, $A$ and $\Omega_{\kappa}A^{\ashk}$ are free as symmetric algebras on a given Lie algebra, with a central element identified with the unit of the symmetric algebra.
The differential and the bracket are both extended from the differential and bracket of $\mathfrak{g}$ as (bi)derivations.
Recall from Examples~\ref{exa:univ-lie} and~\ref{exa:univ-com} the descriptions of the enveloping algebras of Lie algebras and commutative algebras.

\begin{proposition}[Explicit description found in~{\cite[Section~1.1.4]{Fresse2006}}]
  \label{prop:descr-u-pois}
  Let $\mathfrak{g}$ be a $c\Lie$-algebra and $B = S(\Sigma^{1-n} \mathfrak{g})$ the induced $u\Pois_{n}$-algebra.
  Then there is an isomorphism of graded modules:
  \begin{equation}
    \U_{u\Pois_{n}}(B) \cong B \otimes \U_{c\Lie_{n}}(\Sigma^{1-n} \mathfrak{g}).
  \end{equation}
  The algebra $\U_{c\Lie_{n}}(\Sigma^{1-n} \mathfrak{g})$ is generated by symbols $X_{f}$ for $f \in \mathfrak{g}$, with $\deg X_f = \deg f$.
  We have the following relations in $\U_{u\Pois_n}(B)$ (where $f,g \in \mathfrak{g}$, whose suspensions belong to $B \subset \U_{u\Pois_n}(B)$):
  \begin{equation}
    \label{eq:rel-upois}
    \begin{aligned}
      X_{\uni}
       & = 0,
      \\
      X_{fg}
       & = (\Sigma^{1-n}f) \cdot X_{g} + (-1)^{(|f|+1-n) \cdot (|g|+1-n)} (\Sigma^{1-n} g) \cdot X_{f},
      \\
      X_{f} \cdot (\Sigma^{1-n} g)
       & = (\Sigma^{1-n} \{f,g\}) + (-1)^{|f| \cdot (|g|+1-n)} (\Sigma^{1-n}g) \cdot X_{f},
      \\
      X_{\{f,g\}}
       & = X_{f} \cdot X_{g} - (-1)^{|f| \cdot |g|} X_{g} \cdot X_{f}.
    \end{aligned}
  \end{equation}
  In particular, elements of $B$ and $\U_{c\Lie_n}(\Sigma^{1-n}\mathfrak{g})$ do not necessarily commute.
  The differential is the sum of the differential of $B$ and the differential given by $d X_{f} \coloneqq X_{df}$, where $df \in B = S(\Sigma^{1-n} \mathfrak{g})$ and we use the relations to get back to $B \otimes \U(\Sigma^{1-n} \mathfrak{g})$.
\end{proposition}

As explained in~{\cite[Section~1.1.4]{Fresse2006}}, if $M$ is an $u\Pois_n$-module over $B = S(\Sigma^{1-n} \mathfrak{g})$ then $\U_{u\Pois_n}(B)$ acts on $M$ in the following way: an element $b \in B$ act by multiplication, while the element $X_f$ acts by $[\Sigma^{1-n}f, -]$.
The relations above simply encode the Jacobi and Leibniz identities.

\begin{proof}
  The extension of the result from~{\cite[Section~1.1.4]{Fresse2006}} to the unital case is immediate.
\end{proof}

\begin{proposition}
  The derived enveloping algebra and the classical enveloping algebra of the symplectic Poisson $n$-algebra $A_{n;D}$ are quasi-isomorphic.
\end{proposition}
\begin{proof}
  Letting $A = A_{n; D}$, our aim is to prove that $\U_{(u\Pois_{n})_{\infty}}(A)$ and $\U_{u\Pois_{n}}(A)$ are quasi-isomorphic.

  We use the cobar resolution $\Omega_{\kappa}A^{\ashk}$ and the result of Proposition~\ref{prop:derived-u} to obtain that this derived enveloping algebra is quasi-isomorphic to $\U_{u\Pois_{n}}(\Omega_{\kappa}A^{\ashk})$.
  From the description of Proposition~\ref{prop:descr-u-pois}, as a dg-module, this is isomorphic to
  \begin{equation}
    \U_{u\Pois_{n}}(\Omega_{\kappa}A^{\ashk}) \cong \bigl( \Omega_{\kappa}A^{\ashk} \otimes \U_{c\Lie_{n}}(c\Lie_{n}(\Sigma^{-1} \bar{S}^{c}(\Sigma^{n} V))), d_\Omega + d' \bigr),
  \end{equation}
  where $V = \R \langle x_{i}, \xi_{j} \rangle$ is the graded module of generators.
  The product and the generators $X_f \in \U_{c\Lie_n}(c\Lie_{n}(\Sigma^{-1} \bar{S}^{c}(\Sigma^{n} V)))$ of are defined in Equation~\eqref{eq:rel-upois}.
  The differential $d'$ is defined on a generator $X_f$ by $d'(X_f) = X_{df}$ (where we use the relations to get back to $\Omega_\kappa A^{\ashk} \otimes \U_{c\Lie_n}(\dots)$).
  We have a chain map, where $d''$ is defined similarly to $d'$ (but we apply the quotient map $\Omega_\kappa A^{\ashk} \to A$ to the first factor):
  \begin{equation}\label{eq:reduced-complex}
    \U_{u\Pois_{n}}(\Omega_{\kappa}A^{\ashk}) \to \bigl( A \otimes \U_{c\Lie_{n}}(c\Lie_{n}(\Sigma^{-1} \bar{S}^{c}(\Sigma^{n} V))), d'' \bigr)
  \end{equation}
  Let us describe this differential $d''$ explicitly.
  Let $X_{f}$ be a generator of the universal enveloping algebra, for some $f \in \bar{S}^{c}(\Sigma^{n} V)$.
  Then $d''X_{f} = X_{df} = X_{d_{0}f} + X_{d_{2}f}$, where $d_{0}$ and $d_{2}$ were explicitly described in Section~\ref{sec:kosz-prop-expl}.
  In both complexes ($\U_{u\Pois_n}(\Omega_{\kappa}A^{\ashk})$ and the one at the target of \eqref{eq:reduced-complex}), we can filter by the degree of the $\U_{c\Lie_n}$ factor.
  On the first page of the associated spectral sequence, only the $d_\Omega$ differential of the first complex remains, while the differential of the second one vanishes.
  Since $\Omega_{\kappa}A^{\ashk} \to A$ is a quasi-isomorphism, we find that the map~\eqref{eq:reduced-complex} is a quasi-isomorphism.

  Since $d_{0}f$ is a multiple of the unit and $X_{\uni} = \{ \uni, - \} = 0$, we obtain that $X_{d_{0} f} = 0$.
  On the other hand,
  \begin{equation}
    X_{d_{2}f} = \frac{1}{2} \sum_{(f)} (X_{f_{(1)}} X_{f_{(2)}} - (-1)^{|f_{(1)}| \cdot |f_{(2)}|} X_{f_{(2)}} X_{f_{(1)}}),
  \end{equation}
  where we use the shuffle coproduct of $\bar{S}^{c}(X)$.
  Thus we see that the differential stays inside the universal enveloping algebra, and is precisely the one of the bar/cobar resolution of the abelian $c\Lie_{n}$ algebra $V_{+} = V \oplus \R \uni$.
  Thanks to Lemma~\ref{lem:missing} below, we know that $\U_{c\Lie_{n}}$ preserves quasi-isomorphisms (the unit is freely adjoined and hence is not a boundary), and the universal enveloping algebra of an abelian Lie algebra is just a symmetric algebra, hence:
  \begin{equation}
    \U_{u\Pois_{n}}(\Omega_{\kappa}A^{\ashk}) \xrightarrow{\sim} A \otimes \U_{c\Lie_{n}}(V_{+}) \cong A \otimes S(\Sigma^{n-1} V).
  \end{equation}
  This last algebra is simply $\U_{u\Pois_{n}}(A)$, as claimed.
\end{proof}

We now state the missing lemma in the previous proof.
For simplicity we state it for unshifted Lie algebra; the $c\Lie_{n}$ case is identical.
The functor $\U_{\Lie} = \U$ preserves quasi-isomorphisms: we can filter it by tensor powers and apply Künneth's theorem as we are working over a field.
We thus get the following result on $\U_{c\Lie}(-) = \U_{\Lie}(-) / (X_{\uni})$:

\begin{lemma}
  \label{lem:missing}
  The universal enveloping functor $U_{c\Lie}$ preserves quasi-isomorphisms.
\end{lemma}
\begin{proof}
  Let $f : \mathfrak{g} \to \mathfrak{h}$ be a quasi-isomorphism of $c\Lie$-algebras.
  The associative algebra $\U_{c\Lie}(\mathfrak{g})$ is given by the presentation:
  \begin{equation}
    \label{eq:5}
    \U_{c\Lie}(\mathfrak{g}) = T(\mathfrak{g}) / (x \otimes y - (-1)^{|x| \cdot |y|} y \otimes x = [x,y], \uni = 0),
  \end{equation}
  where $T(-)$ is the tensor algebra.
  This algebra is isomorphic to $\U_{\Lie}(\mathfrak{g} / \uni)$.
  Given our quasi-isomorphism $f$, we can apply the five lemma to the following diagram:
  \begin{equation}
    \begin{tikzcd}
      0 \ar[r] & \uni \ar[r, hook] \ar[d, "\cong"] & \mathfrak{g} \ar[r, two heads] \ar[d, "f", "\sim" swap] & \mathfrak{g}/\uni \ar[r] \ar[d, "\bar{f}"] & 0 \\
      0 \ar[r] & \uni \ar[r, hook] & {\mathfrak{h}} \ar[r, two heads] & {\mathfrak{h}/\uni} \ar[r] & 0
    \end{tikzcd}
  \end{equation}
  to obtain that $\bar{f} : \mathfrak{g}/\uni \to \mathfrak{h}/\uni$ is a quasi-isomorphism.
  Since $\U_{\Lie}$ preserves quasi-isomorphisms by the argument above, the result is established.
\end{proof}

\subsection{Factorization homology}
\label{sec:fact-homol-1}

As a second application, let us compute the factorization homology of a fixed parallelized simply connected closed manifold $M$ of dimension $n \geq 4$, with coefficients in $A$ (over $\R$).

\subsubsection{Right \texorpdfstring{$u\Com$}{uCom}-module structure on \texorpdfstring{$\GG_P^\vee$}{G\_P}}
\label{sec:right-ucom-module}

Let $P$ be Poincaré duality model of $M$ (see Section~\ref{sec:fact-homol}) with augmentation $\varepsilon : P^n \to \R$.
Recall from Section~\ref{sec:fact-homol} (or~\cite[Section~5]{Idrissi2016}) that we may use our explicit real model $\GG_{P}^{\vee}$ in order to compute factorization homology $\int_M A$.
The object $\GG_{P}^{\vee}$ is a right $u\Pois_{n}$-module.
As $u\Pois_n = u\Com \circ \Lie_n$, describing the $u\Pois_n$-module structure amounts to describing the $\Lie_n$-module structure and the $u\Com$-module structure.
We are going to describe these two structure separately.

Let us first describe the right $\Lie_n$-module structure of $\GG_P^\vee$, which is found in \cite{Idrissi2016}.
As a right $\Lie_{n}$-module, we have an isomorphism:
\begin{equation}
  \GG_{P}^{\vee} \cong_{\Lie_{n}} C_{*}^{CE}(P^{-*} \otimes \Sigma^{n-1} \Lie_{n}),
\end{equation}
where $C_{*}^{CE}$ is the Chevalley--Eilenberg chain complex, and $\Sigma^{n-1} \Lie_{n} = \{ \Sigma^{n-1} \Lie_{n}(k) \}_{k \ge 0}$ is a Lie algebra in the category of right $\Lie_{n}$-modules.

Let us now describe the right $u\Com$-module structure of $\GG_P^\vee$.
This module structure is not described in~\cite[Section 5]{Idrissi2016}, but it easily follows from the arguments there.
Roughly speaking, one needs to use the distributive law $\Lie_{n} \circ u\Com \to u\Com \circ \Lie_{n}$, which states that the bracket is a biderivation with respect to the product, and that the unit is a central element for the bracket.
Then we either need to use $\varepsilon : P^{n} \to \R$, to describe the action of the unit of $u\Com$, or the coproduct $\Delta : P^{n-*} \to (P^{n-*})^{\otimes 2}$ which is the Poincaré dual of the product $P \otimes P \to P$, to describe the action of the product of $u\Com$.
This coproduct is the unique linear map such that $(\varepsilon_A \otimes \varepsilon_A)(\Delta(a) \cdot (x \otimes y)) = \varepsilon_A(axy)$ for all $a,x,y \in P$.

In more detail, given $k \ge 0$, we have an isomorphism of graded modules:
\begin{equation}
  \GG_{P}^{\vee}(k) \cong \bigoplus_{r \ge 0} \biggl( \bigoplus_{\pi \in \operatorname{Part}_{r}(k)} (A^{n-*})^{\otimes r} \otimes \Lie_{n}(\# \pi_{1}) \otimes \dots \otimes \Lie_{n}(\# \pi_{r}) \biggr)^{\Sigma_{r}},
\end{equation}
where the inner sum runs over all partitions $\pi = \{ \pi_{1}, \dots, \pi_{r} \}$ of $\{1, \dots, k\}$.
To describe the right $u\Com$-module structure, we need to say what happens when we insert the two generators, the unit $\uni$ and the product $\mu$, at each index $1 \le i \le k$, for each summand of the decomposition.

Suppose we are given an element $X = (x_{j})_{j=1}^{r} \otimes \lambda_{1} \otimes \dots \otimes \lambda_{r}$, where $x_{j} \in A$ and $\lambda_{j} \in \Lie_{n}(\# \pi_{j})$.
Suppose that $i \in \pi_{j}$ in the partition.
Then:
\begin{itemize}
  \item $X \circ_{i} \uni$ is obtained by inserting the unit in $\lambda_{j}$:
        \begin{itemize}
          \item if $\lambda_{j}$ has at least one bracket then the result is zero;
          \item otherwise, if $\lambda_{j} = \id$, then the corresponding factor disappears ($c\Lie_{n}(0) = \R$) and we apply $\varepsilon$ to $x_{j}$;
        \end{itemize}
  \item $X \circ_{i} \mu$ is obtained by inserting the product $\mu$ in $\lambda_{j}$.
        Using the distributive law for $\Com$ and $\Lie_{n}$, we obtain a sum of products of two elements from $\Lie_{n}$, splitting $\pi_{j}$ in two subsets.
        We then apply the coproduct $\Delta : P^{n-*} \to (P^{n-*})^{\otimes 2}$, which is Poincaré dual to the product of $P$, to $x_{j}$ to obtain a tensor in $A \otimes A$, which we assign to the two subsets of $\pi_{j}$ in the corresponding summand.
\end{itemize}

\begin{example}
  Given $x \in P$, we can view $x \otimes \id$ as an element of $\GG^\vee_{P}(1) = P^{n-*} \otimes \Lie_{n}(1)$.
  Notice that $\GG_P^\vee(0) = \R$, and $\GG_P^\vee(2) = P^{n-*} \otimes \Lie_n(2) \oplus \bigl( (P^{n-*})^{\otimes 2} \otimes \Lie_n(1)^{\otimes 2} \bigr)^{\Sigma_2}$.
  We then have the following relations:
  \begin{align}
    (x \otimes \id) \circ_{1} \lambda
     & = x \otimes \lambda
    \in P^{n-*} \otimes \Lie_{n}(2).
    \\
    (x \otimes \id) \circ_{1} \mu
     & = \Delta(x) \otimes \id \otimes \id
    \in \bigl( (P^{n-*})^{\otimes 2} \otimes \Lie_{n}(1)^{\otimes 2} \bigr)^{\Sigma_{2}}.
    \\
    (x \otimes \id) \circ_{1} \uni
     & = \varepsilon(x)
    \in \R.
  \end{align}
\end{example}

\subsubsection{Computation}
\label{sec:computation}

\begin{lemma}
  \label{lem:CE}
  The underived relative composition product $\GG_{P}^{\vee} \circ_{u\Pois_{n}} A$ is given by the \emph{unital} Chevalley--Eilenberg homology of the $c\Lie$-algebra $P^{-*} \otimes \Sigma^{n-1} V$.
\end{lemma}

This unital Chevalley--Eilenberg complex is given by
\begin{equation}
  \GG_{P}^{\vee} \circ_{u\Pois_{n}} A = \bigl( S^{c}(P^{-*} \otimes \Sigma^{n} V), d_{CE} \bigr)
\end{equation}
Here the shifted Lie bracket of $V$ (and hence the differential $d_{CE}$) can produce a unit.
In this case, we apply $\varepsilon : P \to \R$ to the corresponding factor, and the result is identified with the counit of $S^{c}(-)$, i.e., the empty tensor.

\begin{proof}
  This is almost identical to the case of the universal enveloping algebra of a Lie algebra (with no central element) from~\cite{Idrissi2016} (see Section~\ref{sec:fact-homol}).
  The Lie bracket cannot produce a product of two elements of $V$, only a unit.
  Therefore we just need to verify what happens to the unit in the isomorphism of~\cite[Lemma~5.2]{Idrissi2016}, which is part of Section~\ref{sec:right-ucom-module}.
\end{proof}

\begin{proposition}
  \label{prop:derived-vs-underived}
  The factorization homology $\int_{M} A \simeq \GG_{P}^{\vee} \circ_{u\Pois_{n}}^{\mathbb{L}} A$ of the symplectic Poisson $n$-algebra $A$ is quasi-isomorphic to $\GG_{P}^{\vee} \circ_{u\Pois_{n}} A$.
\end{proposition}

\begin{proof}
  As we are working with a derived composition product, we take a resolution of $A$ as a $u\Pois_{n}$-algebra.
  For this, we use the cobar construction of the Koszul dual algebra, $\Omega_{\kappa}A^{\ashk}$, described in Section~\ref{sec:kosz-prop-expl}.
  We then have:
  \begin{equation}
    \int_{M} A \simeq \GG_{P}^{\vee} \circ_{u\Pois_{n}} \Omega_{\kappa}A^{\ashk}.
  \end{equation}

  The cobar construction $\Omega_{\kappa}A^{\ashk}$ is a quasi-free $u\Pois_{n}$-algebra on the Koszul dual $qA^{\ashk}$, with some differential.
  Therefore, by the cancellation rule for relative products over operad ($X \circ_{\PP} (\PP \circ Y) = X \circ Y$), we obtain that, as a graded module,
  \begin{equation}
    \GG_{P}^{\vee} \circ_{u\Pois_{n}} \Omega_{\kappa}A^{\ashk} \cong \bigl( \GG_{P}^{\vee} \circ qA^{\ashk}, d_{\Omega} \bigr),
  \end{equation}
  with a differential induced by the differential of the cobar construction.

  The Koszul dual of $qA$ is $qA^{\ashk} = \Sigma^{1-n} \bar{S}^{c}(\Sigma^{n} V)$, where $V = \R \langle x_{i}, \xi_{j} \rangle$ is the graded vector space of generators (see Lemma~\ref{lem:dual-qa}).
  Using the explicit form of the right module $\GG_{P}^{\vee}$ found in Section~\ref{sec:right-ucom-module}, we then find that:
  \begin{equation}
    \int_{M} A \simeq \bigl( S^{c}(P^{-*} \otimes L\bar{S}^{c}(\Sigma^{n} V)), d_{CE} + d_{0} + d_{2} \bigr).
  \end{equation}

  Let us now write down explicit formulas for the three summands of the differential.
  As there are two instances of the cofree cocommutative coalgebra appearing, we have to be careful with our notation.
  We will write $\wedge$ for the tensor of the outer coalgebra, and $\vee$ for the tensor of the inner coalgebra.
  Strictly speaking, we need to consider only elements that are invariant under the symmetric group actions.
  We will consider all elements, and check that formulas are actually well-defined when passing to invariants.
  The three parts of the differentials are:
  \begin{itemize}
    \item Given $x_{1}, \dots, x_{k} \in P$ and $Y_{1}, \dots, Y_{k} \in L\bar{S}^{c}(\Sigma^{n} V)$, we have
          \begin{equation}
            d_{CE}(x_{1} Y_{1} \wedge \dots \wedge x_{k} Y_{k}) = \sum_{i < j} \pm x_{1} Y_{1} \wedge \dots \wedge x_{i}x_{j} [Y_{i}, Y_{j}] \wedge \dots \wedge \widehat{x_{j} Y_{j}} \wedge \dots \wedge x_{k} Y_{k}.
          \end{equation}
    \item The differential $d_{2}$ is defined on the inner $\bar{S}^{c}(\Sigma^{n} V)$, extended to a derivation of $L\bar{S}^{c}(V)$, which is itself extended to the full complex as a coderivation:
          \begin{equation}
            d_{2}(v_{1} \vee \dots \vee v_{k}) = \frac{1}{2} \sum_{\substack{i+j=k\\(\mu,\nu)\in \mathrm{Sh}_{i,j}}} \pm [v_{\mu(1)} \vee \dots \vee v_{\mu(i)}, v_{\nu(1)} \vee \dots \vee v_{\nu(j)}],
          \end{equation}
          where the inner sum is over all $(i,j)$-shuffles.
          (This is the differential of the bar/cobar resolution of the abelian Lie algebra $\Sigma^{n-1} V$).
    \item The differential $d_{0}$ is similarly defined on $\bar{S}^{c}(\Sigma^{n} V)$ and extended to the full complex by (for $X$ a basis element):
          \begin{equation}
            d_{0}(X) =
            \begin{cases}
              -\uni, & \text{if } X = \Sigma^{n} x_{i} \vee \Sigma^{n} \xi_{i} \text{ for some } i; \\
              0      & \text{otherwise}.
            \end{cases}
          \end{equation}
          Note that the unit is appearing here.
          If the unit is inside a Lie bracket, the result is zero ($\uni$ is central).
          Otherwise, we have to apply $\varepsilon : P \to \R$ to the corresponding element of $P$ in the outer $S^{c}(-)$, and this factor disappears (it is replaced with a real coefficient).
  \end{itemize}

  We can project this complex to
  \begin{equation}
    \GG_{P}^{\vee} \circ_{u\Pois_{n}} A = \bigl( S^{c}(P^{-*} \otimes \Sigma^{n} V), d_{CE} \bigr),
  \end{equation}
  i.e., the Chevalley--Eilenberg complex (with constant coefficients) of the $c\Lie$ algebra $P^{-*} \otimes \Sigma^{n-1} V$.
  The projection from $\GG_{P}^{\vee} \circ_{u\Pois_{n}} \Omega_{\kappa}A^{\ashk}$ is compatible with the differential.

  Let $i$ be the number of Lie brackets in an element of the complex, and $j$ be the number of inner tensors ($\vee$).
  We then observe that $d_{2}$ preserves the difference $i-j$, while $d_{CE}$ and $d_{0}$ increase them by one.
  We can thus filter our first complex by this number to obtain what we will call the ``first spectral sequence''.
  The second complex, $\GG_{P}^{\vee} \circ_{u\Pois_{n}} A = \bigl( S^{c}(P^{-*} \otimes \Sigma^{n-1} V), d_{CE} \bigr)$, is also filtered, with the unit in filtration level zero and the rest in filtration level one.
  This yields a ``second spectral sequence''.
  The projection is compatible with this filtration, hence we obtain a morphism from the first spectral sequence to the second one.

  On the $E^{0}$ page of the first spectral sequence, only $d_{2}$ remains.
  Recall that $d_{2}$ is exactly the differential of the bar/cobar resolution $\Sigma^{-1}L\bar{S}^{c}(\Sigma^{n} V) \xrightarrow{\sim} \Sigma^{n-1} V$ of the abelian Lie algebra $\Sigma^{n-1} V$.
  Thus on the $E^{1}$ page of the spectral sequence, we obtain an isomorphism of graded modules from the first spectral sequence to the second.
  The differential $d_{CE}$ of the first spectral sequence vanishes, and the differential $d_{0}$ precisely correspond to part of the ``unital'' Chevalley--Eilenberg differential of the second spectral sequence.
  Hence we find that the projection $\GG_{P}^{\vee} \circ_{u\Pois_{n}} \Omega_{\kappa}A^{\ashk} \to \GG_{P}^{\vee} \circ_{u\Pois_{n}} A$ is a quasi-isomorphism.
\end{proof}

\begin{proposition}
  \label{prop:final-result}
  Let $A$ be a symplectic $n$-Poisson algebra (Definition~\ref{def:sympl-poisson}) and let $M$ be a simply connected smooth framed manifold of dimension at least four.
  Then the homology of $\int_{M} A$ is one-dimensional.
\end{proposition}

\begin{proof}
  Thanks to Proposition~\ref{prop:derived-vs-underived}, we only need to compute the homology of $\GG_{P}^{\vee} \circ_{u\Pois_{n}} A$.
  Let us use the explicit description from Lemma~\ref{lem:CE} as the unital Chevalley--Eilenberg complex  of the $c\Lie$-algebra
  \begin{equation}
    \mathfrak{g}_{P,V} \coloneqq P^{-*} \otimes \Sigma^{n-1} V.
  \end{equation}
  There is a pairing $\langle-,-\rangle : \mathfrak{g}_{P,V}^{\otimes 2} \to \R$ given by $xv \otimes x'v' \mapsto \varepsilon_{P}(xx') \cdot \{ v, v' \}$.
  We have the following isomorphism of chain complexes:
  \begin{equation}
    \GG_{P}^{\vee} \circ_{u\Pois_{n}} A \cong \biggl( \bigoplus_{i \ge 0} \bigl( (\Sigma \mathfrak{g}_{P,V})^{\otimes i} \bigr)_{\Sigma_{i}}, d_{CE} \biggr)
  \end{equation}
  where
  \begin{equation}
    d_{CE}(\alpha_{1} \wedge \dots \wedge \alpha_{k}) = \sum_{i < j} \pm \langle \alpha_{i}, \alpha_{j} \rangle \cdot \alpha_{1} \wedge \dots \widehat{\alpha}_{i} \dots \widehat{\alpha}_{j} \dots \wedge \alpha_{k}.
  \end{equation}

  Recall that since $P$ is a Poincaré commutative dg-algebra, the pairing induced by $\varepsilon_A$ is by definition non-degenerate.
  Moreover, the pairing on $V$ given by the bracket is clearly non-degenerate.
  It follows that the pairing $\langle -, -\rangle$ defined on $\mathfrak{g}_{P,V}$ is non-degenerate.
  If $\{a_i\}_{i \in I}$ is a graded basis of $P$ and $\{a_i^\vee\}_{i \in I}$ is its dual basis, then $\{ a_i \otimes x_j, a_i \otimes \xi_j \}_{i \in I, 1 \leq j \leq D}$ is a graded basis of $\mathfrak{g}_{P,V}$ and its dual basis is $\{ a_i^\vee \otimes \xi_j, a_i^\vee \otimes x_j \}_{i \in I, 1 \leq j \leq D}$.

  To lghten up the notation, let us write $\{ y_k \}_{1 \leq k \leq r}$ for the graded basis of $\mathfrak{g}_{P,V}$ found above and let $\{ y_k^\vee \}$ be its dual basis under the pairing.

  We can then identify $\GG_{P}^{\vee} \circ_{u\Pois_{n}} A$ with the ``algebraic de Rham complex'':
  \begin{equation}
    \Omega^{*}_{adR}(\R^{r}) = \bigl( S(y_{1}, \dots, y_{r}) \otimes \Lambda(dy_{1}, \dots, dy_{r}), d_{dR} = \sum_{k} \frac{\partial}{\partial y_{k}} \cdot dy_{k} \bigr).
  \end{equation}
  Note that if all the variables $y_k$ had degree zero then this would be isomorphic to the algebra $A_{PL}(\Delta^{r}) \otimes_{\Q} \R$ of piecewise polynomial real forms on $\Delta^{r}$.
  There is an isomorphism (up to a degree shift and reversal) given by:
  \begin{equation}
    \begin{aligned}
      \biggl( \bigoplus_{i \ge 0} \bigl( (\Sigma \mathfrak{g}_{P,V})^{\otimes i} \bigr)_{\Sigma_{i}}, d_{CE} \biggr)
       & \xrightarrow{\cong} \Omega^{*}_{adR}(\R^{r})
      \\
      y_{k_{1}} \wedge \dots \wedge y_{k_{\alpha}} \wedge y_{l_{1}}^{*} \wedge \dots \wedge y_{l_{\beta}}^{*}
       & \mapsto y_{k_{1}} \dots y_{k_{\alpha}} \cdot \prod_{\mathclap{\substack{1 \le l \le r \\ l \not\in \{l_{1}, \dots, l_{\beta}\}}}} dy_{l}
    \end{aligned}
  \end{equation}
  For example if $r=3$, then the isomorphism sends $y_{1} \wedge y_{2}^{*}$ to $y_{1} dy_{1} dy_{3}$, i.e., the isomorphism sends a form to its Hodge star complement.

  The algebraic de Rham complex is a particular example of a Koszul complex and is therefore acyclic.
  There is an explicit homotopy given by $h(dy_{i}) = y_{i}$, $h(y_{i}) = 0$ and extended suitably as a derivation.
  In particular, a representative of the only homology class is the unit of the de Rham complex, which under our identification is $\bigwedge_{j = 1}^{r} y_{j}^{*}$.
\end{proof}

\begin{remark}
  From a physical point of view, this result is satisfactory: when one wants to compute expected values of observables, one wants a number.
  The next best thing to a number is a closed element in a complex whose homology is one-dimensional.
  We thank T.~Willwacher for this perspective.
\end{remark}

\begin{remark}
  \label{rem:markarian}
  The above proof appears to be similar to the computation of Markarian~\cite{Markarian2017} for the Weyl $n$-algebra $\mathcal{W}_{n}^{h}(D)$, which is an algebra over the operad $C_{*}(\mathsf{FM}_{n}; \R[[h, h^{-1}]])$, where $\mathsf{FM}_{n}$ is the Fulton--MacPherson operad.
  However, we do not know the precise relationship between $A_{n;D}$ and $\mathcal{W}_{n}^{h}(D)$.
  Curved Koszul duality was conjectured to apply for this computation by Markarian~\cite{MarkarianTanaka2015}.
  Moreover, Döppenschmitt~\cite{Doeppenschmitt2018} recently released a preprint containing an analogous computation for a twisted version of $A$, using a ``physical'' approach based on AKSZ theory.
  Our approach is however different from these two approaches.
  It is also in some sense more general, as we should be able to compute the factorization homology of $M$ with coefficients in any Koszul $u\Pois_{n}$-algebra, e.g., an algebra of the type $S(\Sigma^{1-n} \mathfrak{g})$ where $\mathfrak{g}$ is a Koszul $c\Lie$-algebra.
  Finally, let us note that previous results of Getzler~\cite{Getzler1999}, regarding the computation of the Hodge polynomials (with compact support) of configuration spaces of quasi-projective varieties over a base, involve similar techniques.
\end{remark}

\begin{remark}
  Using the results from~\cite{CamposIdrissiLambrechtsWillwacher2018}, we hope to be able to compute factorization homology of compact manifolds with boundary with coefficients in $A$.
\end{remark}

\printbibliography
\end{document}